\documentclass[12pt]{article}

\usepackage[margin=1.2in]{geometry}
\usepackage{lipsum}
\usepackage{amsfonts}
\usepackage{graphicx}
\usepackage{epstopdf}
\usepackage[ruled,vlined]{algorithm2e}

\usepackage{color}
\usepackage{subcaption}
\usepackage{amsmath}
\usepackage{amssymb}
\usepackage{enumitem}
\usepackage{bbm}
\usepackage{bm}
\usepackage{hyperref}
\usepackage{verbatim}
\usepackage{float}
\usepackage{booktabs}
\usepackage{ulem}

\usepackage{xcolor}

\usepackage{cleveref}

\usepackage{lineno}
\nolinenumbers

\newtheorem{proposition}{Proposition}[section]

\newtheorem{proof}{Proof}[section]
\newtheorem{remark}{Remark}

\usepackage{amsopn}

\def\*#1{\boldsymbol{#1}}

\newcommand{\R}{\mathbb{R}}

\newcommand{\E}{\mathbb{E}}

\newcommand{\Fcal}{\mathcal{F}}

\newcommand{\Xcal}{\mathcal{X}}
\newcommand{\Ycal}{\mathcal{Y}}

\newcommand{\Ocal}{\mathcal{O}}

\newcommand{\f}{\boldsymbol{f}}
\newcommand{\y}{\boldsymbol{y}}

\newcommand{\kk}{\boldsymbol{k}}
\newcommand{\ttheta}{\boldsymbol{\theta}}
\newcommand{\dd}{\mathrm{d}}

\title{Personalized Algorithm Generation: \\
A Case Study in Learning ODE Integrators\thanks{Submitted to the SIAM Journal on Scientific Computing (SISC)’s Methods and Algorithms for Scientific Computing section May 17, 2021; accepted for publication (in revised form) February 23, 2022; published electronically July 7, 2022.

https://doi.org/10.1137/21M1418629

Funding: DTD and MD received funding from the Helmholtz Association of German Research Centres and performed
    this work as part of the Helmholtz School for Data Science in Life, Earth and Energy (HDS-LEE).
    The work of FD, TB and IGK was partially supported through an ARO MURI (Dr. M. Munson) and the DARPA ATLAS program (Dr. J.Zhou).
    FD is also supported by the DFG (Emmy Noether project 468830823).
    QL is supported by the National Research Foundation, Singapore,
    under the NRF fellowship (NRF-NRFF13-2021-0005).
}}

\date{}

\author{
Yue Guo\thanks{Department of Mathematics, National University of Singapore, 117543, Singapore.}
\and
Felix Dietrich\thanks{Institut für Informatik, TU München, Boltzmannstr. 3, 85748 Garching b. München, Germany.}
\and
Tom Bertalan\thanks{Department of Chemical and Biomolecular Engineering, Whiting School of Engineering, Johns Hopkins University, 3400 North Charles Street, Baltimore, MD 21218, USA.}
\and
Danimir T. Doncevic \thanks{Institute of Energy and Climate Research -- Energy Systems Engineering (IEK-10), Forschungszentrum J\"ulich GmbH,  52425 J\"ulich, Germany; RWTH Aachen University, Aachen 52062, Germany.}
\and
Manuel Dahmen\thanks{Institute of Energy and Climate Research -- Energy Systems Engineering (IEK-10), Forschungszentrum J\"ulich GmbH,  52425 J\"ulich, Germany.}
\and
Ioannis G. Kevrekidis\thanks{Department of Chemical and Biomolecular Engineering and
Department of Applied Mathematics and Statistics,
Whiting School of Engineering, Johns Hopkins University,
3400 North Charles Street, Baltimore, MD 21218, USA.}
\and
Qianxiao Li\thanks{Department of Mathematics, National University of Singapore, 117543, Singapore; Institute of High Performance Computing, A*STAR, 138632, Singapore.
\textbf{Corresponding author:} qianxiao@nus.edu.sg.}
}

\begin{document}

\maketitle

\begin{abstract}
                We study the learning of numerical algorithms for scientific computing, which combines mathematically driven, 
        handcrafted design of general algorithm structure with a data-driven adaptation to specific classes of tasks.
        This represents a departure from the classical approaches in numerical analysis, which typically do not feature such learning-based adaptations.
        As a case study, we develop a machine learning approach that automatically learns effective solvers for initial value problems in the form of ordinary differential equations (ODEs), based on the Runge-Kutta (RK) integrator architecture.
        We show that we can learn high-order integrators for targeted families of differential equations without the need for computing integrator coefficients by hand.
        Moreover, we demonstrate that in certain cases we can obtain superior performance to classical RK methods. This can be attributed to certain properties of the ODE families being identified and exploited by the approach.
        Overall, this work demonstrates an effective learning-based approach to the design of algorithms for the numerical solution of differential equations.
        This can be readily extended to other numerical tasks.
\end{abstract}
\section{Introduction}
\label{sec:intro}

In computational mathematics, one is interested in developing solvers for different types of problems, such as algebraic equations, differential equations or optimization problems. In an abstract setting, these can be written as
\begin{equation}
    L(y, F) = 0,
\end{equation}
where $y \in \Ycal$ is the unknown, $F \in \Fcal$ is the problem instance and $L : \Ycal \times \Fcal \rightarrow \Xcal$ is a mapping representing the problem type. The sets $\Ycal,\Fcal,\Xcal$ are usually some subsets of normed vector spaces and can be finite or infinite dimensional, depending on the application.

Here are some examples: 
\begin{itemize}
    \item
    An algebraic equation $f(y) = 0$ can be recast as
    \begin{equation*}
        L(y, F) = f(y) = 0,
    \end{equation*}
    where $F=f$.
    \item
    An optimization problem $\min_{y} f(y)$ can be recast as
    \begin{equation*}
        L(y, F) = f(y) - \min_z f(z) = 0,
    \end{equation*}
    where $F=f$.
    \item
    A differential equation $\dot{y}(t) = f(y(t))$, $y(0)=y_0$ on $t\in[0,T]$ can be recast as
    \begin{equation*}
        L(y, F)(t) = y(t) - y_0 - \int_{0}^{t} f(y(s)) ds = 0,
    \end{equation*}
    where $F$ comprises the ODE's information related to vector field $f$ and initial condition $y_{0}$.
    Note here that the extrinsic input is $(f, y_0)$ and the output of $L$ is a function of time.
\end{itemize}

For a fixed problem type $L$, a solution operator is a mapping $A : \Fcal \rightarrow \Ycal$,
which produces the true solution $y = A(F)$ given a problem instance $F$, so that $L(y, F) = 0$.
Often, we do not have an explicit means to represent $A$.
Thus, for computational purposes we design a numerical algorithm that computes an estimate solution
$
    y \approx \hat{A}(F,h),
$
where $h>0$ denotes the accuracy of approximation.
We call
$
    \hat{A} : \Fcal \times \R_+ \rightarrow \Ycal
$
an approximate solver, which is consistent if $\lim_{h\rightarrow0} \hat{A}(\cdot, h) = A(\cdot)$.
In this work, we also consider parametric approximate solvers 
$
    \hat{A} : \Fcal \times \R_+ \times \Theta \rightarrow \Ycal
$
where $\Theta$ is a set of solver parameters that can be optimized according to problem settings.




Classical numerical methods design the solver $\hat{A}(\cdot,h)$ by requiring it to perform well {\it over a large and, in general, unstructured class} $\Fcal$. For example, one might seek
\begin{equation}\label{eq:abstract_sup}
    \sup_{F\in\Fcal}
    \| L
        (
            \hat{A}(F,h), F
        )
    \|
    = \Ocal(h^{\alpha}),
    \qquad
    \alpha > 0.
\end{equation}
However, often in practice we are not interested in
such a worst-case approach. In fact, we may want to solve a special class of problems belonging to $\Fcal$ (e.g. only integrate symplectic ODEs), and we may only be interested in the average performance of our method on this class of problems. Hence, instead of~\eqref{eq:abstract_sup}, we may require
\begin{equation}\label{eq:abstract_avg}
E_{\mu}[\| L
        (
            \hat{A}(F,h), F
        )
    \|]
    =\int_{F \in \Fcal}
    \| L
        (
            \hat{A}(F,h), F
        )
    \|
    d\mu(F)
    = \Ocal(h^{\alpha}),
    \qquad
    \alpha > 0,
\end{equation}
where $\mu$ is a probability measure on $\Fcal$ and may be supported on a very small subset.
This imparts structure in $\Fcal$ through $\mu$, and our algorithm is now only required to perform well in expectation under this structure.
In other words, we want to find an approximate solver that is adapted to a restricted problem class.

One can think of this as ``personalized" algorithm: one tuned to the class of problems,
    over the range of parameters of interest.
    This goes beyond algorithm {\it parameter} tuning, to possibly include algorithm structure design and also combinations of
    multiple algorithms.
Success over large and unstructured classes or problems would, of course, allow the algorithm (and the associated code) to be portable across many  physical models, and
has long been an obvious advantage for scientific computation
    - both for learning how to perform it, and for performing it.
For specific applications there are almost always some tuning involved: integrators for stiff vs. nonstiff problems; symplectic integrators for Hamiltonian vs. ``general" nonsymplectic ones, lower vs. higher order optimization algorithms, etc.
Yet, tuning the algorithm to the specific problem was left to the practitioner interested in the specific problem: Which algorithm? What order? What accuracy? How frequent the adaptation? - and has been mostly done ``by hand".
This is not surprising, since incorporating complex and varied structures into algorithms requires a detailed understanding
of the structure of the problem at hand, and often has to be treated on a case by case basis.
However, machine learning allows us to contemplate delegating this task to the computer: we need to first choose the class of problems of interest, and then devise
sufficiently general superstructures (in this case, superstructures for neural network architectures), that will perform the personalized tuning.
For example, we can parameterize the approximate solver as a neural network
$\hat{A}(\cdot,\cdot; \ttheta)$ where $\ttheta \in \Theta$ is a vector of fitting parameters,  that can be determined
from training using appropriate error metrics over a range of tasks of interest.
The parametrization $\hat{A}(\cdot, \cdot ; \ttheta)$ represents an architecture where $\ttheta$ specifies how the algorithm can operate on the $F$ to produce an approximate solution. The optimal way that $F$ is used to produce the solution will depend on the structure of the problem induced by $\mu$, and machine learning can help us find an approximately optimal way to do so.
This is also a form of {\it multi-task learning} in the broad sense, since we want a solver that performs well on not just one {\it task} $F$, but on a distribution of tasks \cite{Hospedales2020metalearning}.
On the algorithmic side, it also shares similarities with the MAML algorithm \cite{finn2017model} in meta-learning (see discussion in \Cref{sec:related_work}).
This also connects with the use of optimization (e.g., a Hamilton-Jacobi-Bellman approach) for the generation of optimal algorithms
(See \cite{traub1980general}, \cite{traub1988information} and \cite{li2017stochastic}).
Such a use of optimization over superstructures for optimal algorithm generation has been recently proposed and illustrated in 
\cite{Mitsos2018algorithms}; we will return to this in \Cref{sec:discussion}. 

In this paper, we will investigate a particular realization of this general problem by studying ODE integrators, which are approximate solvers for some initial value problems. In particular, we develop a learning-based algorithm to generate effective and specialized integrators adapted to specified problem settings.

The rest of the paper is organized as follows. In \Cref{sec:formulation}, we formulate the precise problem of learning integrators. Next, in \Cref{sec:Model Architecture and Choice of Loss Functions} we introduce our architecture to achieve this, based on the Runge-Kutta family of integrators, and our learning algorithm based on losses derived from Taylor expansions. In \Cref{sec:results}, we demonstrate the effectiveness of our learned integrators on selected benchmark function families and provide some analysis to understand the origin of the improvement over classical methods. We conclude with discussions on related work in \Cref{sec:related_work}, together with some general observations and future directions in \Cref{sec:discussion}.

\section{Problem Formulation for Case Study}
\label{sec:formulation}

We focus on a particular realization of the general problem we discussed before: learning high-accuracy integrators adapted to integrating specific families of ordinary differential equations (ODEs).
We begin by introducing the background and basic notions of ODE integrators, with particular emphasis on the Runge-Kutta family of explicit integrators, which form the basis of our neural network parameterization. We conclude this section with the precise mathematical formulation of our integrator learning problem.

\subsection{Ordinary Differential Equations and Integrators}
\label{subsec:ode_integrators}

Consider a time-homogeneous ordinary differential equation describing an initial value problem in $\R^d$
\begin{align}\label{eq:ode}
    \frac{\dd}{\dd t} \y(t) = \f(\y(t)),
    \qquad
    \f : \R^d \rightarrow \R^d,
    \qquad
    \y(0) = \y_{0} \in \R^d.
\end{align}
Here, $\f$ is a vector field driving the evolution equation.
Instead of one specific $\f$, we will consider a family $H$ of vector fields. For simplicity, we will assume that $H$ contains only Lipschitz-continuous functions, so that~\cref{eq:ode} admits a unique solution. Note that~\cref{eq:ode} includes as a special case time-inhomogeneous equations $\dd \y/\dd t  = \f(t, \y)$, since we may always define an additional variable $\tau(t)$ such that $\dd \tau/\dd t = 1$ and redefine $\tilde{\y} = (\tau, \y)$ and $\tilde{\f}(\tilde{\y}) = [1, \f(\tau, \y)]$. 
The only caveat is that this redefinition requires $\f$ to be Lipschitz in $t$, whereas for general ODE theory this condition can be relaxed \cite{iserles2009first}. Nevertheless, for numerical computation such a technical issue is less important, and thus we will hereafter only consider the time-homogeneous case without loss of generality.

For general $\f$, \cref{eq:ode} does not admit an explicit closed-form solution, and one often resorts to a numerical approximation via a solver.
Let $F=(\f, \y_{0})$ define a problem instance and the algorithm parameter $h$ represents a desired level of precision of the solution, then an integrator builds an approximate solution iteratively. In the simplest case of explicit, one-step integrators, 
one iterates the following formula based on an integrator $I_{\hat{A}}$ that computes 
\begin{align}\label{eq:difference}
    \hat{\y}_{n+1} = I_{\hat{A}}(\f, \hat{\y}_{n}, h),
    \qquad
    \hat{\y}_{0} = \y_0.
\end{align}
This produces an approximate sequence $\hat{\y}_n \approx \y(nh)$.
In fact, we can understand the mapping from $F=(\y_0, \f)$ to 
a continuous-time interpolation of $\{ (nh, \hat{\y}_{n}) \}$
as a solver $\hat{A}(\cdot, h) : \mathcal{F} \rightarrow \mathcal{Y}$.

The accuracy of the integrator is measured by the {\it local} and {\it global} truncation errors.
We write the solution of \cref{eq:ode} with $t=nh$ as $\y_{n} := y(nh)$.
The local truncation error is defined as the one-step error between the integrator and the true solution, i.e.
\begin{align}
    E_{h,1}(\y_0) = \| \hat{\y}_{1} - \y_{1} \|.
\end{align}

The integrator is called {\it consistent} if $E_1(\y_0) = o(h)$ for each $\y_0 \in \R^d$.
On the other hand, the global truncation error is
\begin{align}\label{evalu_formula}
    E_{h,n}(\y_0) = \| \hat{\y}_{n} - \y_{n} \|.
\end{align}
The integrator is said to be {\it convergent} if $\lim_{h\rightarrow 0} \max_{m\leq n} E_{h,m}(\y_0) = 0$.

One has finer measures of performance in terms of the {\it order of convergence}. In particular, we say that a convergent integrator is of global order $p>0$ if
\begin{align}
   \max_{m\leq n} E_{h,m}(\y_0) = \mathcal{O}(h^p).
\end{align}$$$$

\subsection{Explicit Runge-Kutta Integrators}
\label{subsec:ERK}

We use the prediction from the RK method with targeted order to ensure the order of our method. In general, our approach is not restricted to the explicit RK family, and other integration methods like multi-stage DAE solvers, etc. are also possible.
Now, let us introduce the family of integrators known as explicit Runge-Kutta integrators (ERK)~\cite{iserles2009first}.
While these methods are well known, we give a brief account here in order to motivate subsequent developments in our learning-based approach, which depends on the structures of RK integrators.

Let us write the solution of the ODE \cref{eq:ode} as
\begin{align}
\y\left(t_{n+1}\right)=\y\left(t_{n}\right)+\int_{t_{n}}^{t_{n+1}} \f(\y(\tau)) \mathrm{d} \tau=\y\left(t_{n}\right)+h \int_{0}^{1} \f\left(\y\left(t_{n}+h \tau\right)\right) \mathrm{d} \tau.
\end{align}
Then, an approximate solution can be found by applying quadrature to the last integral
\begin{align}\label{rk_formula}
\y_{n+1} \approx \y_{n}+h \sum_{i=1}^{m} b_{i} \f\left(\y\left(t_{n}+c_{i} h\right)\right), \quad n=0,1, \ldots
\end{align}
It remains to approximate $\y\left(t_{n}+c_{i} h\right)$ by vectors $\boldsymbol{\xi}_{i}, i=1,2, \ldots, m.$
We set $c_{1}=0$, then $\boldsymbol{\xi}_{1}=\y_{n}$.
The idea behind {\it explicit Runge-Kutta (ERK)} methods is to express each $\boldsymbol{\xi}_{i},i=2,3, \ldots, m,$ by updating $\y_{n}$ with a linear combination of $\f\left(\boldsymbol{\xi}_{1}\right),\ldots, \f\left(\boldsymbol{\xi}_{i-1}\right) .$
This leads to the integrator
\begin{align}\label{rk_derive}
\begin{split}
\boldsymbol{\xi}_{m} &=\y_{n}+h \sum_{j=1}^{m-1} a_{m, j} \f\left(\boldsymbol{\xi}_{j}\right), \\
\hat{\y}_{n+1} &=\y_{n}+h \sum_{i=1}^{m} b_{i} \f\left(\boldsymbol{\xi}_{i}\right).
\end{split}
\end{align}
Appropriate choices of the coefficients $\{ a_{i,j}, b_i \}$ then ensures that our approximation is accurate to the desired level.
To determine these coefficients, we expand and equate the Taylor series of $\y_{n+1}$ with that of $\hat{\y}_{n+1}$ about $\y_{n}$.
Note that the conditions do not define an ERK integrator uniquely, and any choices of the coefficients, from which the same order accuracy can be obtained, are considered as ERK methods.

\subsection{Learning Integrators}

Because a large number of equations need to be solved to obtain the coefficients, deriving high-order ERK integrators is nontrivial. Moreover, the order of convergence $p$ is forced upon all $(p+1)$-times continuously differentiable functions, which may be a much larger family than what one might be interested in integrating in practice. Here, we explore the following: Can one obtain better integrators adapted to a smaller, structured family of problems $\mathcal{F}$?

To this end, we parameterize a family of approximate solvers
\begin{align}
    \mathcal{A}(\Theta)
    := \{ \hat{A}(\cdot, \cdot; \ttheta) : \ttheta \in \Theta \},
\end{align}
%
where $\Theta$ is some subset of a Euclidean space representing the fitting parameters (trainable weights).

Let $\mu$ be a probability measure on $\mathcal{F}$, representing a particular distribution of tasks $F$. 
Let $\mathcal{L} : \R^{d} \times \R^{d} \rightarrow \R_+$ be a loss functions which is minimized when its first two arguments are equal.
Then, we consider the following optimization problem
\begin{align}
\label{eq:learning_formulation}
\begin{split}
    \min_{\ttheta \in \Theta} \quad
    &
    \E_{F \sim \mu, h \sim \nu}
    \left[
        \mathcal{L}(
            \y_{n},
            \hat{\y}_{n})
    +\mathcal{R}(\hat{A}(\cdot, \cdot;\ttheta), F, h)
    \right]\\
    s.t. \quad
    &
    \y_{n} = A(F)(nh) = \y_{0} + \int_{0}^{nh} \f(\y(s)) ds,\\
    &
    \hat{\y}_{n} = \hat{A}(F, h; \ttheta)(nh) 
    = I_{\hat{A}}(\f, \hat{\y}_{n-1}, h; \ttheta),\\
    & F=(\f, \y_{0}),\\
    & n \geq 0.
 \end{split}
\end{align}
Here we define a formula in the last term
$\mathcal{R}(\hat{A}(\cdot, \cdot;\ttheta), F, h)$,
which is independent of the choice for $h$.
We just consider the situation near $t=0$,
only depending on the structure of the task $F$.
This represents a regularization term that allows us to promote certain order of accuracy, and we shall discuss its importance in \cref{subsec:loss}.

Problem \cref{eq:learning_formulation} is the central formulation of this paper, where we rephrase the problem of finding an effective integrator as an optimization problem. Note that this formulation takes explicit account of the fact that the problem instance $F = (\f, \y_0)$ is not generic, but rather belongs to a potentially structured function class $\mathcal{F}$ endowed with a probability measure $\mu$, representing the distribution of tasks. Moreover, note that we also consider a measure over the step sizes $h \sim \nu$, indicating the fact that we are not always looking for integrators that work equally well for all step sizes. Finally, we note that~\cref{eq:learning_formulation} is a population risk minimization problem in the language of machine learning, and hence to solve it we often need to replace the respective expectations by averages over samples from the respective probability measures. In the next section, we will discuss the parameterization of RK-like integrators using neural networks and the choice of loss functions and regularizers that enables one to solve~\cref{eq:learning_formulation} to yield novel integrators.
\section{Model Architecture and Choice of Loss Functions}
\label{sec:Model Architecture and Choice of Loss Functions}

In this section, we outline our method for solving~\cref{eq:learning_formulation}. We begin with the parameterization of the family of solvers $\mathcal{A}(\Theta)$ 
using neural networks, after which we introduce the crucial choice of loss functions and regularizers which enable us to learn accurate integrators.

\subsection{NN Architectures Parameterizing RK-like Integrators}
\label{rk_nn_arch}

Recall that the Runge-Kutta (RK) family of integrators form a sequential linear combination of function evaluations to build the integrator via approximate quadrature. 
Here, we can build a neural network that parameterizes the general form of RK-type integrators.
The integrator is constructed as $\kk_{1} = \f\left(\mathrm{id}(\hat{\y}_{n})\right)$ where $\mathrm{id}$ is the identity function,
$\kk_{i}=h \f\left(\hat{\y}_{n}+\sum_{j=1}^{i-1} \theta_{i-1, j} \kk_{j}\right)$ for $i = 2,\ldots, m$, and $\hat{\y}_{n+1}=\hat{\y}_{n}+\sum_{i=1}^{m} \theta_{ci} \kk_{i}$.
The neural network architecture based on the above formulation is given in the left side of \cref{fig:rk_meta}.
\begin{figure}[htb!]
	\centering
	\includegraphics[width=1.0\textwidth]{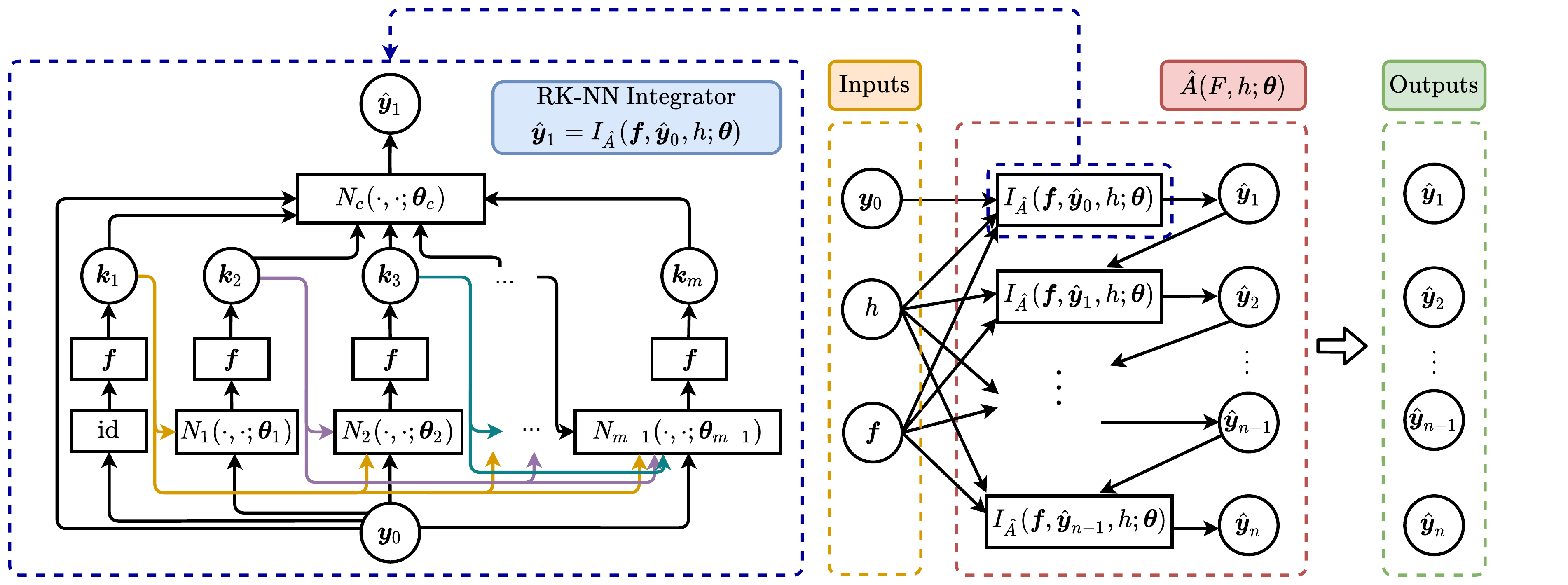}
	\caption{RK-like Neural Network (RK-NN) Architecture.}
	\label{fig:rk_meta}
\end{figure}

There are two types of trainable components in the RK-like neural network (RK-NN):
\begin{enumerate}
	\item Submodels $N_i, i=1,2,\ldots, m-1$, provide inputs to the function $\f$, given $\hat{\y}_t$, $\kk$ from the last step, time step $h$ and parameters $\ttheta_{i} = (\theta_{i,1},\theta_{i,2},\ldots, \theta_{i,i})$.
	\item Combined model $N_c$ has output as $\hat{\y}_{n+1}$, given $\hat{\y}_n$, all the variables $\kk_m$, time step $h$ and parameters $\ttheta_c = (\theta_{c1},\theta_{c2},\ldots, \theta_{cm})$.
\end{enumerate}
Each of these components is represented by a linear layer.
If we set $m = 3$, the architecture represents a generalized RK3 integrator.

Recall the similar architecture of the RK-m method, where
each $N_{i}(\cdot, \cdot; \ttheta_{i})$ is constructed by $\hat{\y}_{n} + \sum_{j=1}^{i}\theta_{i,j} \kk_{j}$.
$\theta_{i,j} \in \R$, and $\hat{\y}_{n}$, $\kk_{j}$ are $d$-dimensional vectors.
The trainable parameters are the coefficients before each $\kk_{j}$.
The transformation in the last layer $N_{c}$ of the whole RK network architecture has the form $\hat{\y}_{n} + \sum_{i = 1}^{m}\theta_{ci}\kk_{i}$.
Thus, the RK-NN parameterization retains the linear combination steps present
in classical RK methods.
Alternatively, the RK-NN can be viewed as a feed-forward neural network with linear layers ($N_{i}$ and $N_{c}$) and the following modifications.
First, we use the vector field $\f$ of an ODE as a nonlinear module that can be interpreted as a custom, vector-valued activation function on the layers.
Second, the layers are not sequential but connected to all the previous $\kk_j$'s.
This viewpoint connects the RK-NN architecture with techniques on exploiting integrators to learn continuous dynamical systems, i.e., the module $\f$, from data. See \Cref{sec:related_work} for further discussion of this point.

Finally, to make sure that the sum of $\theta_{ci}$ is equal to 1, we apply a softmax activation
$
\theta_{ci}=\frac{e^{z_{i}}}{\sum_{j=1}^{m} e^{z_{j}}} \text { for } i=1, \ldots, m
$,
where $z_{i} \in \R$ are the trainable variables in the final layer.
We use softmax to allow for direct application of unconstrained optimization methods.
The downside is that we always have positive values for the coefficients,
which cannot be exactly 0 or 1.
However, the softmax activation can approach 0 or 1 quickly,
and we found in practice that this did not pose a problem.
The total number of trainable parameters in RK-NN is $m+\sum_{i=1}^{m-1} i = \frac{m(m+1)}{2}$.


An important point in the choice of architectures is that we want the integrators to be consistent, i.e., the local truncation error should vanish as $h\rightarrow 0$. This is ensured by our parametric construction, as shown below.

\begin{proposition}
\label{prop:rk_nn_consistent}
	For any $\ttheta \in \Theta$, the NN parameterization in RK-NN (shown in \cref{fig:rk_meta}) is consistent.
\end{proposition}
\begin{proof}
	By the definition of explicit RK method and the Taylor expansion of $\f$ at $\hat{\y}_{n}$, we have $\kk_{1}=h\f(\hat{\y}_{n})$ and $
	\kk_{i}=h\f(\hat{\y}_{n}+\sum_{j=1}^{i-1}\theta_{i-1,j}\kk_{j})
			=h\f(\hat{\y}_{n})+o(h)$ for $i = 2,\ldots, m$.
Then the predicted value is $\hat{\y}_{n+1} = \hat{\y}_{n}+\sum_{i=1}^{m}\theta_{ci}\kk_{i}=\hat{\y}_{n}+\sum_{i=1}^{m}\theta_{ci}h\f(\hat{\y}_{n})+o(h)$.
The sum of $\theta_{ci}$ is equal to 1 since 
$
\theta_{ci}=\frac{e^{z_{i}}}{\sum_{j=1}^{m} e^{z_{j}}} $.
Thus, we obtain $\hat{\y}_{n+1} = \hat{\y}_{n}+h\f(\hat{\y}_{n})+o(h)$.
The taylor expansion of the true solution is 
$\tilde{\y}\left(t_{n+1}\right)
=\hat{\y}_{n}
+h \f\left(\hat{\y}_{n}\right)
+o(h)
$.
We observe that coefficients before the first order of h are the same, then $E_1(\hat{\y}_n) = \| \hat{\y}_{n+1} - \tilde{\y}\left(t_{n+1}\right) \|=o(h)$,
which shows this integrator is consistent.
\end{proof}

\subsection{Choice of Loss Function and Regularizer}
\label{subsec:loss}

To solve~\cref{eq:learning_formulation}, we need to define the loss function $\mathcal{L}$ and the regularizer $\mathcal{R}$. 

\paragraph{Loss function}
We make the simple choice of a scaled square loss
\begin{align}
\label{loss}
	\mathcal{L}(
            \y_{n},
            \hat{\y}_{n})
            = \frac{ \| \y_{n} - \hat{\y}_{n} \|^2}{\| \y_{n} - \hat{\y}^{(RK)}_{n} \|^2},
\end{align}
where $F=(\f, \y_{0})$,
$\y_{n} = A(F)(nh)$,
    $\hat{\y}_{n} = \hat{A}(F, h; \ttheta)(nh) 
    = I_{\hat{A}}(\f, \hat{\y}_{n-1}, h; \ttheta)$ and $\hat{\y}^{(RK)}_{n} = \hat{A}_{RK}(F, h)(nh) 
    = I_{\hat{A}_{RK}}(\f, \hat{\y}^{(RK)}_{n-1}, h)$.
$\hat{\y}$ is the prediction from our RK-NN integrator and $\hat{\y}^{(RK)}$ is from the RK method.
Here, we consider one-step prediction by setting $n=1$.
We use the difference between the RK prediction and the true solution
as a scale stabilization numerics, since the errors tend to be small for small $h$.
If we expect our RK-NN method to be trained to a specific order $\alpha$,
RK-$\alpha$ is chosen in \cref{loss} correspondingly.
In the case where the true solution $\y_1 \equiv \y(h)$ is not known, we can compute its Taylor expansion
near $h=0$ up to appropriate order and use it as a surrogate.
Note that computing the Taylor expansion only requires $\f$:
\begin{subequations}\label{soln_ODE_y}
\begin{align}
\label{soln_ODE_y_truncated}
	\y(h)&=\y(0)
	 + \sum_{m=1}^{n}\frac{1}{m!}\frac{\dd^{m} \y(0)}{\dd h^{m}} h^{m}
	 + \mathcal{O}\left(h^{n+1}\right),\\
	 \label{deriv_first}
		\frac{\dd \y}{\dd h}&= \f (\y),\\
		\label{deriv_element}
		\frac{\dd^{m} \y}{\dd h^{m}} &= (\frac{\dd^{m} y_{1}}{\dd h^{m}},\ldots, \frac{\dd^{m} y_{d}}{\dd h^{m}})^{T}, \quad \frac{\dd^{m} y_{i}}{\dd h^{m}}=\sum_{j=1}^{d} \frac{\partial\left(\frac{\dd^{m-1} y_{i}}{\dd h^{m-1}}\right)}{\partial y_{j}} f_{j},
\end{align}
\end{subequations}
The appropriate order of the computed Taylor expansion depends on the desired integrator accuracy.
For example, to obtain a third-order integrator, $n\geq3$ in \cref{soln_ODE_y_truncated} is chosen as a surrogate of the true solution and we choose RK3 as the reference algorithm.

\paragraph{Regularizer}

The loss function alone cannot ensure that we can achieve a desired order of accuracy, since the mean squared loss has vastly different contributions from different values of $h$.
To overcome this issue, we introduce a regularizer that
promotes high order of convergence of the global truncation error over a span of integration step sizes.
Recall that to obtain an integrator with $\mathcal{O}(h^{\alpha})$ global error,
we need the local truncation error to be
$
	\mathcal{O}(h^{\alpha+1}).
$
This is achieved by ensuring
$
    \frac{d^{i}}{d h^{i}}\big|_{h=0}\left(\y_{1}-\hat{\y}_{1}\right)=0, \forall i=1, \ldots, \alpha,
$
or equivalently,
$
	\sum_{i=1}^{\alpha}\|\frac{d^{i}}{d h^{i}}|_{h=0}\left(\y_{1} - \hat{\y}_{1}\right)\|_{2}^2 = 0.
$
The latter is scalar-valued, thus convenient to turn into a regularizer
\begin{align}
\label{taylor_loss}
	\mathcal{R}(\hat{A}(\cdot, \cdot;\ttheta), F, h)
	=
	\sum_{i=1}^{\alpha}\left\|\frac{d^{i}}{d h^{i}}\bigg|_{h=0}\left(\y_{1} - \hat{\y}_{1}\right)\right\|_{2}^2,
\end{align}
which promotes the desired order of convergence.

In implementation, the derivatives can be evaluated at exactly $h=0$ through automatic differentiation
in Tensorflow.
In the case where the true solution is not known and we use the Taylor expansion surrogate \cref{soln_ODE_y_truncated},
automatic differentiation can still be applied at $h=0$ to the surrogate.

\subsection{Learning algorithm}

Having defined the loss functions and regularizers, it remains to train the network using standard stochastic gradient methods, with sample means to approximate the expectations in~\cref{eq:learning_formulation}.
The performance of the RK-NN integrator is quantified by the relative error $\gamma$ compared with the reference RK method, $\gamma < 1$ implies an improvement.
The entire learning algorithm is summarized in Alg.~\ref{alg:learning_algo}.
The code is open source and can be found at \url{https://github.com/GUOYUE-Cynthia/Learning-ODE-Integrators}.

\begin{algorithm}
	\SetAlgoLined
	\SetKwInOut{Req}{Require}
	\SetKwInOut{Init}{Initialize}
	\SetKwFor{ForAll}{for all}{do}{end}
	\KwData{$\mathcal{D}=\{F_{j}, h_{j}\}_{j=1}^{N}$;}
	\Init{Random $\ttheta_{0}$ for operator $\hat{A}(\cdot, \cdot; \ttheta_{0}): \ttheta_{0} \in \Theta , h>0$; \\ Set tolerance $\epsilon >0$; Optimizer \texttt{Opt};}
	\For{$k = 0, 1, \ldots, \# Iterations$}{
	\ForAll{$F_{j}, h_{j}$}{
	Calculate $\y_{n}^{(j)} = A(F_{j})(nh_{j})$;\\
	Calculate $\hat{\y}_{n}^{(j)} = \hat{A}(F_{j}, h_{j}; \ttheta)(nh_{j})$;\\
	Calculate $\hat{\y}^{(RK)(j)}_{n} = \hat{A}_{RK}(F_{j}, h_{j})(nh_{j})$;\\
	Calculate the scaled loss: $\mathcal{L}(\y_{n}^{(j)}, \hat{\y}_{n}^{(j)})$;\\
	Calculate the regularizer:
	$\mathcal{R}(\hat{A}(\cdot, \cdot;\ttheta), F_{j}, h_{j})$;
	}
	Evaluate 
    $\ell = 
    \frac{1}{N}\sum_{j=1}^{N}
    \left[
        \mathcal{L}(
            \y_{n}^{(j)},
            \hat{\y}_{n}^{(j)})
    +\mathcal{R}(\hat{A}(\cdot, \cdot;\ttheta), F_{j}, h_{j})
    \right]$;
	\\
	Update parameters $\ttheta$ using $\texttt{Opt}$ to minimize $\ell$;\\
	Compute the relative error: $\gamma = \frac{1}{N}\sum_{j=1}^{N}\mathcal{L}(
            \y_{n}^{(j)},
            \hat{\y}_{n}^{(j)})$;\\
    \If{$\gamma < \epsilon$}{\textbf{break};}}
	\KwRet{Operator $\hat{A}(\cdot, \cdot; \ttheta).$}
	\caption{Learning Algorithm.}
	\label{alg:learning_algo}
\end{algorithm}


\begin{remark}[Computational complexity and memory usage]
In terms of inference (i.e. integrating a new ODE using the trained RK-NN), computational complexity and memory usage are the same as the traditional RK method.
This is because the difference between RK-NN and traditional RK methods lies only 
in the value of the learned coefficients.
During training, the memory cost of storing weights is independent of the dimension of the ODE,
because the number of parameters in RK-NN is determined by the number of RK stages.
The computational complexity of training is more delicate, and depends on a variety of factors,
including the form of the ODE family, the software implementation of back-propagation, and hardware optimization.
Instead of theoretical estimates,
we show empirically (See supplementary material, \textbf{SM4}) that
training RK-NN integrator on an example linear family and a nonlinear family
has favorable scaling between $O(d)$ and $O(d^{2})$ as the number of dimensions $d$ increases.
Thus, it can be applied to moderately sized integration problems.
\end{remark}

\section{Results}
\label{sec:results}

We now present the results of RK-NN on various test problems,
with particular emphasis on how the learned integrators differ fundamentally from the classical RK methods.
In all experiments, we use the Adam optimizer \cite{kingma2014adam} to train the RK-NN.
To check the accuracy order,
we use the error defined in \cref{evalu_formula}, hereafter abbreviated as $E$,
as the evaluation standard to quantify the global error.
Note that the order of accuracy from different integrators can be inferred from the slopes of the log-scale plots.

\subsection{Learning High-order Integrators}

We first demonstrate that the learning algorithm indeed learns high-order integrators for several test problem settings.
In all experiments, the integration time step $h$ used for training is sampled uniformly in $(0.01,0.1)$.

\paragraph{Linear Task Family}
The simplest task family is the pairs of stable linear functions and initial conditions $(\f, \y_{0}) \in \Fcal$, which has the form
\begin{align}
\begin{split}
    \Fcal &= 
    \{ \y \mapsto -a\y \mid a > 0\} \times \{ \R \}, \\
    \mu &= \mathrm{Distribution}(\{\y \mapsto -a\y; a \sim U(1, 5)\}) \times U(-5, 5).
\end{split}
\end{align}
In this case, the closed-form solution is
$
\boldsymbol{y}(t) = e^{-a t} \boldsymbol{y}_{0}.
$

\paragraph{Square Task Family}
$\f$ is a scaled element-wise square function $f(\y)_i = - a y_i^{2}$ and $\y_{0} \sim U(1,3)$, thus $\Fcal$ has the form
\begin{align}
\begin{split}
    \Fcal &= 
    \{ \y \mapsto -a\y^2 \mid a > 0\} \times \{ \R \},\\
    \mu &= \mathrm{Distribution}(\{\y \mapsto -a\y^2; a \sim U(0.1, 0.5)\}) \times U(1, 3).
\end{split}
\end{align}
The true solution is
$
\boldsymbol{y}(t) = (at+ 1/ \boldsymbol{y}_{0})^{-1}.
$


\Cref{fig:rk3_order3} compares the performance of the learned integrator to that of RK3. For the sake of comparison, we will set $m=3$ in our RK-like neural network (see \cref{fig:rk_meta}). Moreover, we set $\alpha=3$ in the regularizer to promote a similar global truncation order as RK3 integrator.
We observe from \cref{fig:rk3_order3} that we can indeed learn integrators that 
out-performs
the RK3 method when $h$ is in the training range, but becomes worse when extending to a larger range of $h$. 
This is significant since we did not need to compute the Butcher Tableau explicitly, and the coefficients of the RK-NN are chosen automatically via machine learning. This makes the method easily scalable to higher order methods, unlike explicitly derived RK integrators which can become very complex. 
It is also evident that the learned RK-NN is different from a usual RK3 method computed from the Butcher Tableau.
Furthermore, to validate the importance of the regularizer defined in \cref{taylor_loss},
we train RK-NN without the regularizer and compare its performance in \cref{fig:rk3_order3}.
We observe that in this case, the global accuracy depends more sensitively on $h$,
and deteriorates rapidly outside the range of $h$ used for training.


\begin{figure}[htb!] 
\centering
	\begin{subfigure}{0.49\textwidth}
		\centering
		\includegraphics[width=1.0\linewidth]{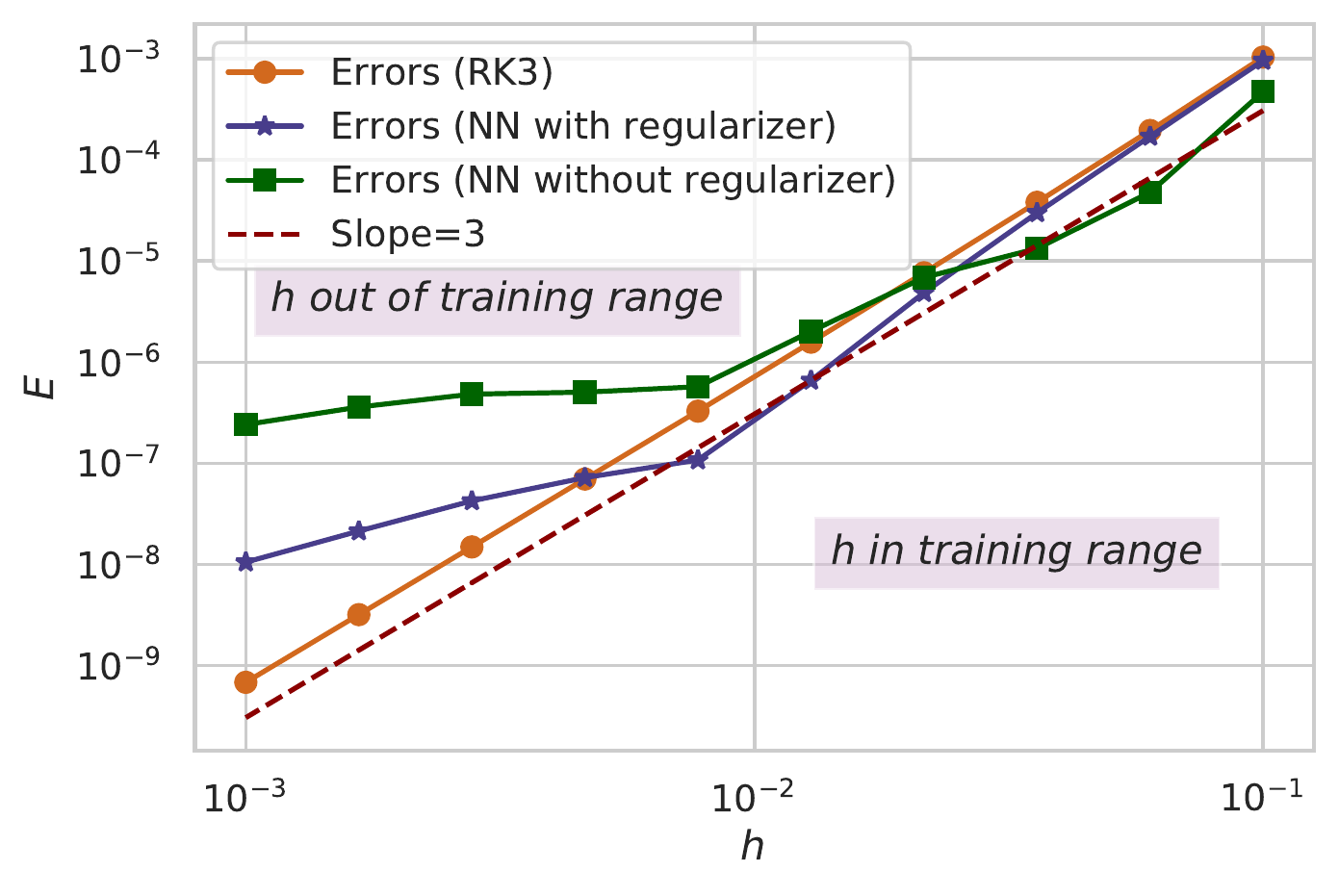}
		\caption{Linear Task Families.}
	\end{subfigure}
	\begin{subfigure}{0.49\textwidth}
		\centering
		\includegraphics[width=1.0\linewidth]{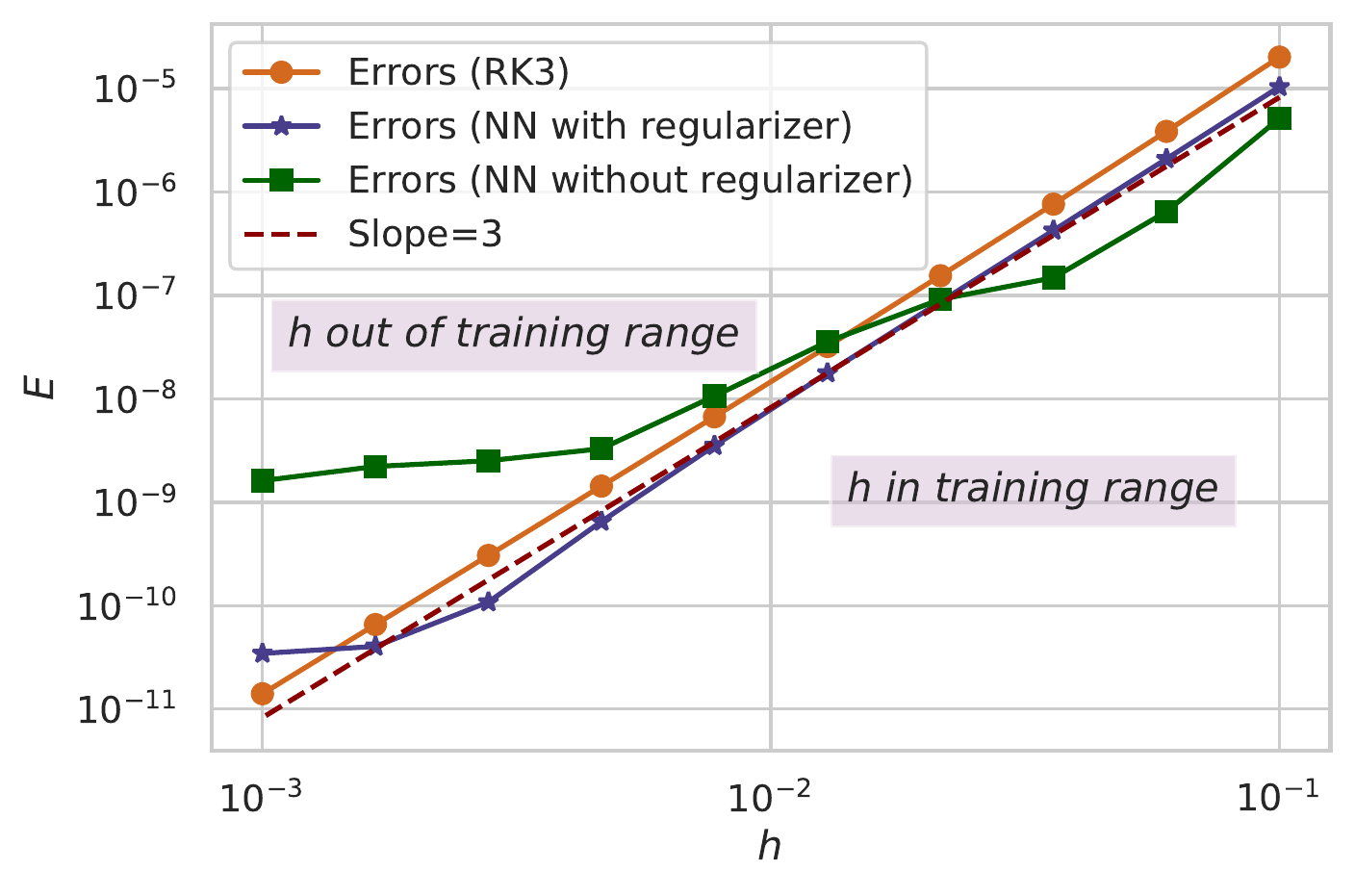}
		\caption{Square Task Families.}
	\end{subfigure}
	\caption{Error comparison among RK3 method,  \textbf{three-stage} RK-NN with and without \textbf{third-order} Taylor-based loss as regularizer.
	The training time step range is  $h \in (0.01, 0.1)$ whereas the testing range extends to
	 $h \in (0.001, 0.1)$.
    }
	\label{fig:rk3_order3}
\end{figure}


\subsection{Generalization Across Task Families}

From the learning formulation, it is clear that our goal is to maximize the performance of our integrator on a specific task family $\Fcal$. Hence, it is expected that the effectiveness of the learned integrators may not generalize across different families. 
\Cref{fig:train_test_diff} shows that this is indeed the case.
This is expected, since the learned integrator is adapted to the family of integration tasks that it is trained on. Nevertheless, note that in all cases the learned integrator maintains consistency, which is ensured by the construction of the parametric NN family (See \cref{prop:rk_nn_consistent}). {\it This shows in particular that we have not learned a generic RK3 integrator.}

\begin{figure}[htb!]
\centering
	\begin{subfigure}{0.49\textwidth}
		\centering
		\includegraphics[width=1.0\linewidth]{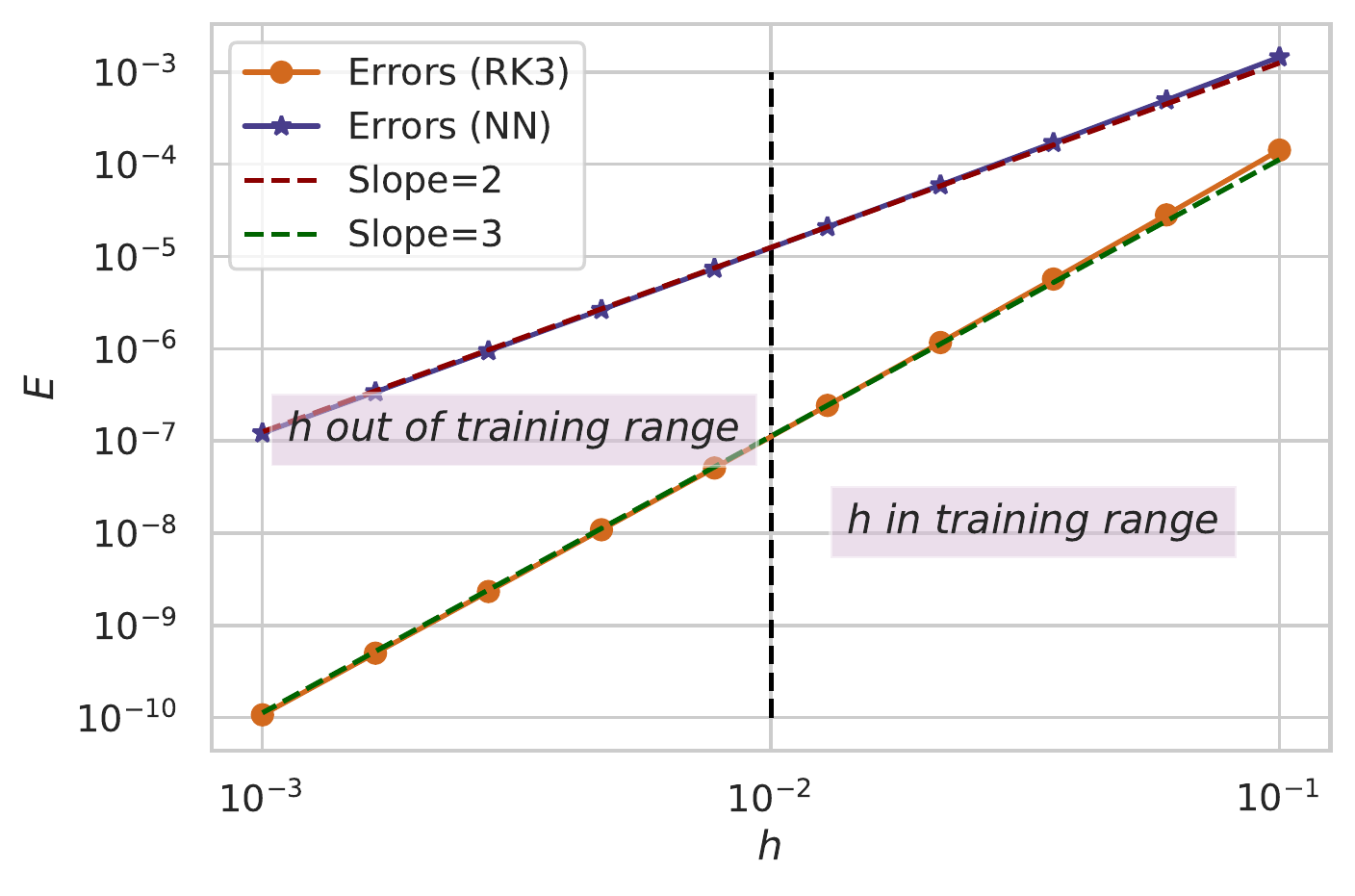}
		\caption{Train on linear task families but test on square task families.}
		\label{fig:train_linear_test_nonlinear}
	\end{subfigure}
	\begin{subfigure}{0.49\textwidth}
		\centering
		\includegraphics[width=1.0\linewidth]{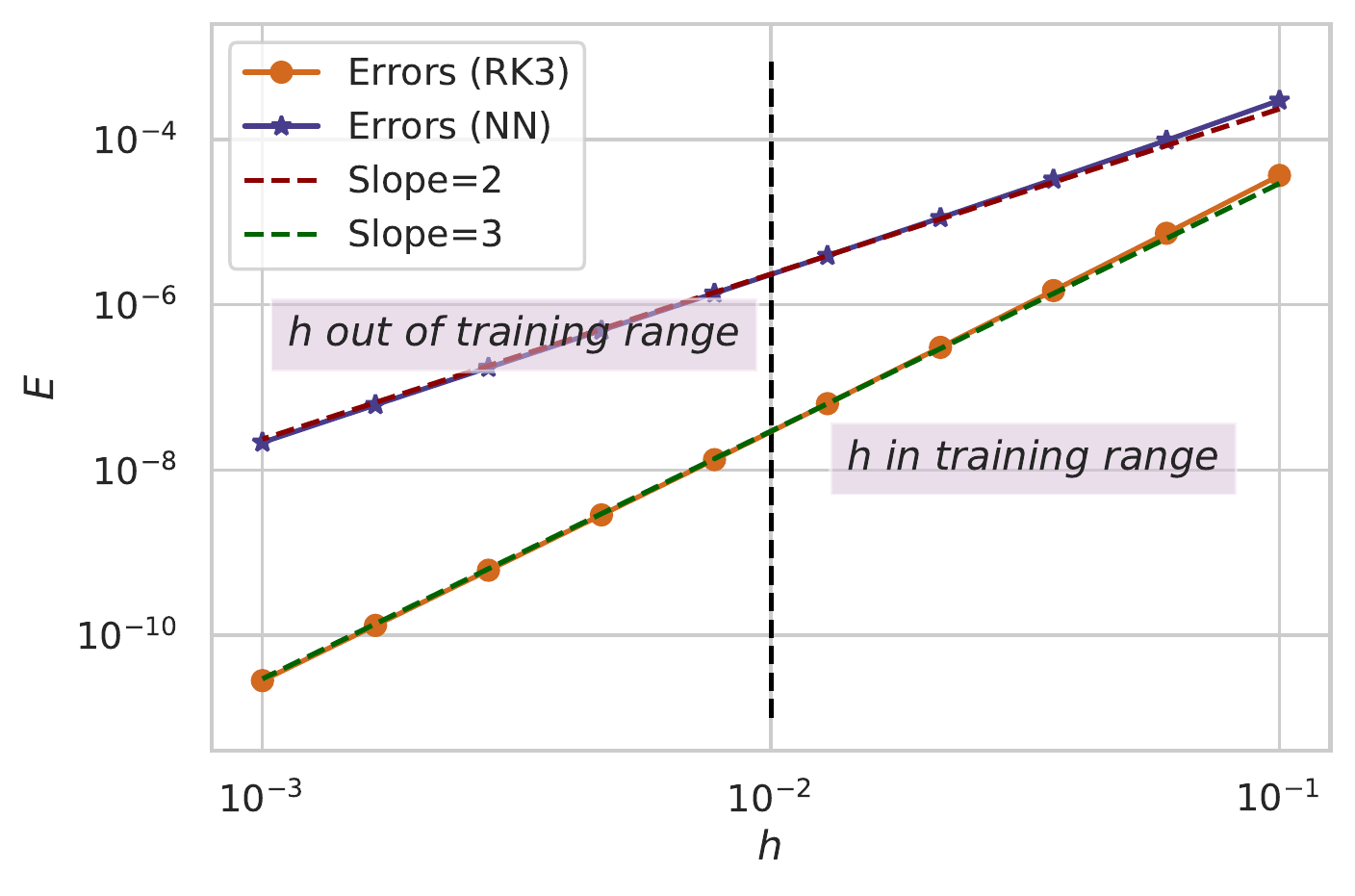}
		\caption{Train on square task families but test on linear task families.}
		\label{fig:train_nonlinear_test_linear}
	\end{subfigure}
	\caption{Training and testing on different task families. Time steps $h$ in these plots are in the range $(0.001, 0.1)$, which includes values both inside and outside the training range.}
	\label{fig:train_test_diff}
\end{figure}

Although the limited generalization ability may appear to be a limitation of the approach, it turns out that this enables us to obtain very efficient integrators that out-performs classical RK integrators on such restricted families. The next part discusses this point in detail.

\subsection{Outperforming Classical RK Integrators}
\label{Outperforming}
In general, if an explicit $m$-stage Runge–Kutta method has order $p$, then it can be proven that the number of stages must satisfy $m\geq p$. In particular, for $p\geq 5$ one can show that we require $m\geq p+1$, so no five-stage RK method can reach order 5 global truncation accuracy~\cite{butcher2016numerical}.
However, we now show that if we limit our attention to a specific task family, our approach can learn integrators that overcome this limitation, at least within a limited range of integration step-sizes. 

\paragraph{Training a two-stage RK-NN integrator to have third-order accuracy}
Here, we set $m=2$ (two-stage RK-NN), but we set $\alpha=3$ in the regularizer, which promotes a third-order accuracy. The results are shown in~\cref{fig:rk2_order3} for the nonlinear (square) task family.
We observe that we can indeed learn a two-stage RK-NN that has third-order accuracy (the slope of the RK-NN error is 3) for this task family.
This is not possible for the general class of Lipschitz functions,
for which a two-stage method can only achieve second-order global accuracy.


\begin{figure}[htb!]
	\centering	\includegraphics[width=0.6\textwidth]{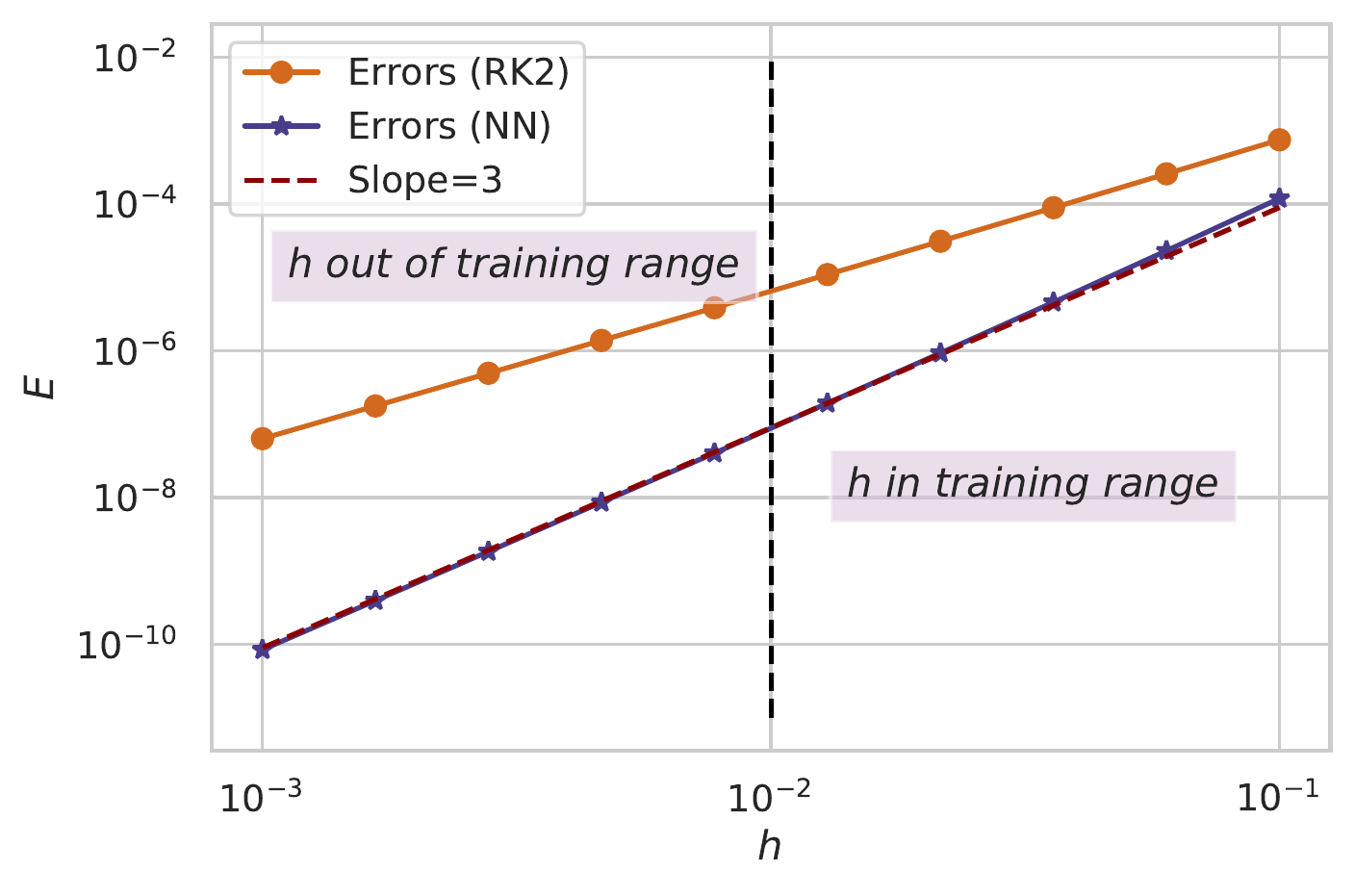}
	\caption{Error analysis on \textbf{square task families}, training on $h\in (0.01,0.1)$ but testing on $h\in (0.001,0.1)$, using \textbf{two-stage} RK-NN integrator with \textbf{third-order} Taylor-based loss as the regularizer.}
	\label{fig:rk2_order3}
\end{figure}

For this specific family, we can in fact understand this phenomenon precisely.
Recall the two-stage RK method for equation \cref{eq:ode}:
\begin{align}
		\kk_{1} = h f(\boldsymbol{y}_{n}), \quad
		\kk_{2} = h f(\boldsymbol{y_{n}} + \theta_{1} \kk_{1}), \quad
		\boldsymbol{y}_{n+1} = \boldsymbol{y}_{n} + \theta_{c1} \kk_{1} + \theta_{c2} \kk_{2},
\end{align}
where $\theta_{1}, \theta_{c1}, \theta_{c2}$ are constants.

For the square task family with $d=1$ , the right hand side of equation \cref{eq:ode} is
 \begin{align}
 \frac{\dd}{\dd t} y(t) = f(y) = - a y^{2},
 \qquad
 y(0) = y_{0} \in \R.
 \end{align}
Then we can obtain
\begin{align}\label{rk2}
	\begin{split}
	k_{1} &= - a y_{n}^{2}h, \quad
	k_{2} = - a y_{n}^{2} h
	       + 2 \theta_{1} a^{2} y_{n}^{3} h^{2}
	       -  \theta_{1}^{2} a^{3} y_{n}^{4} h^{3},\\
	y_{n+1}
	&= y_{n}-\left(\theta_{c1} + \theta_{c2}\right) a y_{n}^{2} h
	+ 2\theta_{1}\theta_{c2} a^{2} y_{n}^{3} h^{2}
	- \theta_{1}^{2} \theta_{c2} a^{3} y_{n}^{4} h^{3}.
	\end{split}
\end{align}

We expand the true solution at $t_{n+1}$,
 subject to the initial condition $y_{n}$ at $t_{n}$.
Due to Taylor's theorem,
\begin{align}\label{taylor_expan_rk2}
	\begin{split}
		\widetilde{y}(t_{n+1})
		&= y_{n}
		+ y^{\prime}_{n} h
		+ \frac{1}{2}y^{\prime \prime}_{n} h^{2}
		+ \frac{1}{6}y^{\prime \prime \prime}_{n} h^{3}
		+ \mathcal{O}(h^{4})\\
		&= y_{n}
		- a y_{n}^{2} h
		+ a^{2} y_{n}^{3} h^{2}
		- a^{3} y_{n}^{4} h^{3}
		+ \mathcal{O}(h^{4}).
	\end{split}
\end{align}
Comparison between \cref{rk2} and \cref{taylor_expan_rk2} shows the conditions for obtaining a third-order integrator by a two-stage RK method:
\begin{align}\label{rk2_taylor3_conditions}
	 \theta_{c1} + \theta_{c2} = 1, \quad 2\theta_{1}\theta_{c2} = 1, \quad \theta_{1}^{2} \theta_{c2} = 1. 
\end{align}
Note that this relies on the fact that we are only considering the family of vector fields of the form $\{- a y_n^2\}$.
In general, it is not possible to obtain third-order accuracy only with a two-stage RK-type integrator. Indeed, the coefficients in the learned RK-NN, 
\begin{align}
\theta_{1}=2, \quad \theta_{c1}=0.75, \quad \theta_{c2}=0.25,
\end{align}
are consistent with conditions in \cref{rk2_taylor3_conditions},
while the standard RK2 coefficients,
\begin{align}
\theta_1^\mathrm{RK}=1,\quad\theta_{c1}^\mathrm{RK}=0.5,\quad\theta_{c2}^\mathrm{RK}=0.5,
\end{align}
are not.

In this example, two-stage RK-NN can provably achieve third-order accuracy, due to the special structure in the task family. In general, one will not expect this to hold for all values of $h$, especially those far out of the training regime. 
Similarly, we also obtain a four-stage RK-NN integrator having sixth-order accuracy when training and testing over $h \in (0.01, 0.1)$, which is shown in supplementary material \textbf{SM1}.



\subsection{Training a RK-NN Integrator on the Van der Pol Oscillator and the Brusselator}
Up to now, we studied the performance of our RK-NN method on 1-dimensional tasks. 
Now we apply our method to multivariate instances.

\paragraph{Van der Pol Oscillator}
\begin{align}
\frac{\dd^{2} u}{\dd t^{2}}-a\left(1-u^{2}\right) \frac{\dd u}{\dd t}+u=0,
\end{align}
where the parameter $a$ is a scalar parameter indicating the nonlinearity and the strength of the damping.
Another commonly used form based on the transformation $v = \dot u $ leads to:
\begin{align}
\begin{split}
\dot{u}&=v, \\
\dot{v}&=a\left(1-u^{2}\right) v-u.
\end{split}	
\end{align}
For the purposes of our method, $\y$ is vector-valued with $\y=(u,v)$.

\paragraph{The Brusselator}
\begin{align}
\begin{split}
\dot{u} &=1-(b+1) u+a u^{2} v, \\
\dot{v} &=b u-a u^{2} v,
\end{split}
\end{align}
where $a, b>0$ are constants and $u, v \in \mathbb{R}$. $u$ and $v$ represent the dimensionless
concentrations of two of the reactants.
Same as the notation in the Van der Pol oscillator, the dynamic variable is also the vector $\y=(u,v)$ in our method.

Conducting experiments on the mentioned 2-dimensional tasks,
we can obtain an integrator better than the traditional RK method on specific ODE problems. 
The highlight is that our trained integrator brings lower error between the true solution and prediction value than the RK-method with the same accuracy order.

It is worth considering why we do not directly use Taylor series with the same order truncation error as our integrator, now that we can leverage the ability to get the differential of $\y$ or $h$ during our training.
When it comes to a third-order algorithm,
 the integrator based on Taylor series is defined as
$\y(h)=\y(0)+\sum_{i=1}^{3}\frac{1}{i!}\y^{(i)}(0) h^{i}$.
We will show the comparison among our RK-NN method, RK3 method and Taylor series integrator in the following experiments.

\cref{fig:vdp_bru_ordercheck} shows that the learned integrator for each task can achieve third-order accuracy respectively on the Van der Pol oscillator and the Brusselator families in the training range.
To make the comparison clear,
we plot the geometric mean of the relative errors with quantified uncertainties with respect to the RK method. 

\begin{figure}[htb!]
\centering
	\begin{subfigure}{0.49\textwidth}
		\centering
		\includegraphics[width=1.0\linewidth]{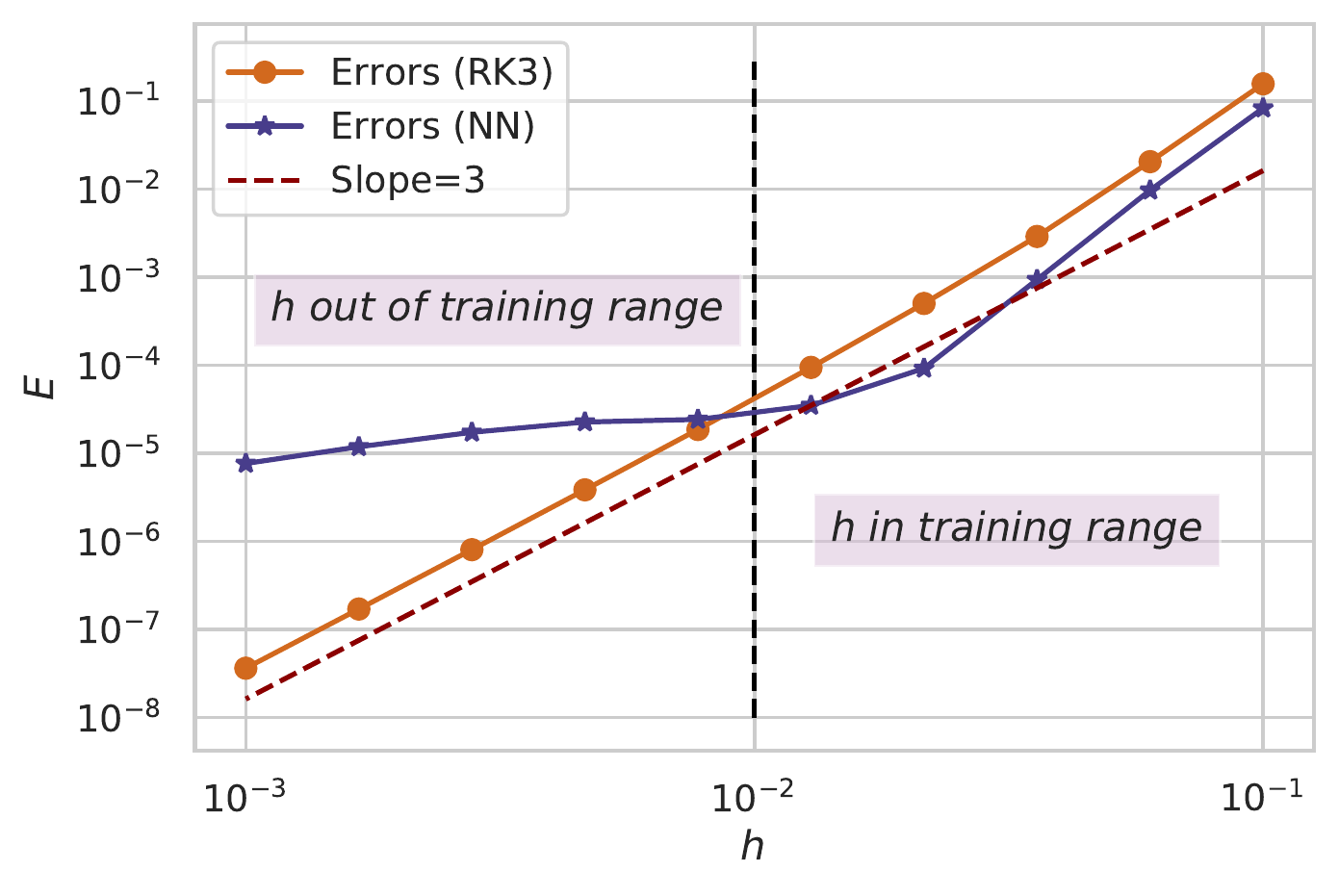}
		\caption{Van der Pol oscillator.}
		\label{fig:vdp_ordercheck}
	\end{subfigure}
	\begin{subfigure}{0.49\textwidth}
		\centering
		\includegraphics[width=1.0\linewidth]{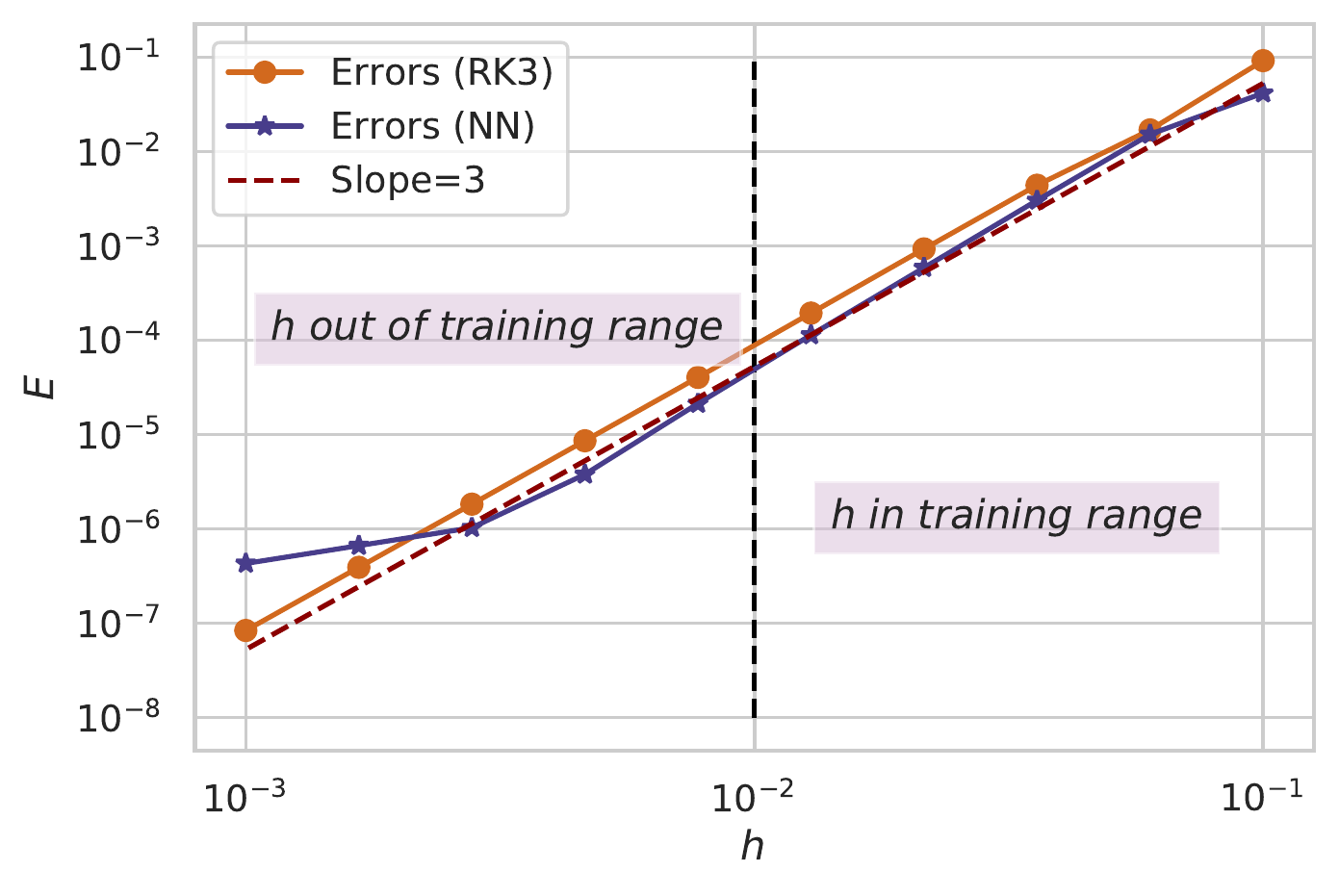}
		\caption{The Brusselator.}
		\label{fig:bru_ordercheck}
	\end{subfigure}
	\caption{Order Check after training for a \textbf{third-order} integrator by \textbf{three-stage} RK-NN integrator.}
	\label{fig:vdp_bru_ordercheck}
\end{figure}	


For Van der Pol oscillator families,
the inputs are sampled with $a \sim U (1,2)$ as vector field parameters and $\{\y_{0}=(u_{0},v_{0}); u_{0} \sim U(-4,-3), v_{0} \sim U(0,2)\}$ as initial conditions. 
Better performance than with traditional RK integrators can be obtained within our selected training range, but the performance can become worse when extrapolating outside of the training range. For instance, \cref{fig:vdp} demonstrates that the learned integrators are adapted to the training range of step sizes $h$, but their convergence order decreases as $h$ attains smaller values.
Further, when we test the RK-NN outside of the training range of the parameter $a$, we observe ambiguous behavior. For $a \in (0,1)$, \cref{vdp_alpha_0_1}, the accuracy of the trained RK-NN becomes worse, whereas the extrapolation to $a \in (2,3)$ in \cref{vdp_alpha_2_3} behaves similarly to $a$ in the training range, i.e., leading to improved performance as long as $h$ is also within the training range. Similar behavior occurs when testing with initial values $(u_0, v_0)$ outside of the training range, \cref{fig:vdp_y0_outside}. 
Altogether, the observations in Figures 6-8 indicate that our RK-NN reliably outperforms traditional RK integrators within the training range, but is not necessarily accurate outside it.


\begin{figure}[htb!]
\centering
	\begin{subfigure}{0.49\textwidth}
		\centering
		\includegraphics[width=1.0\linewidth]{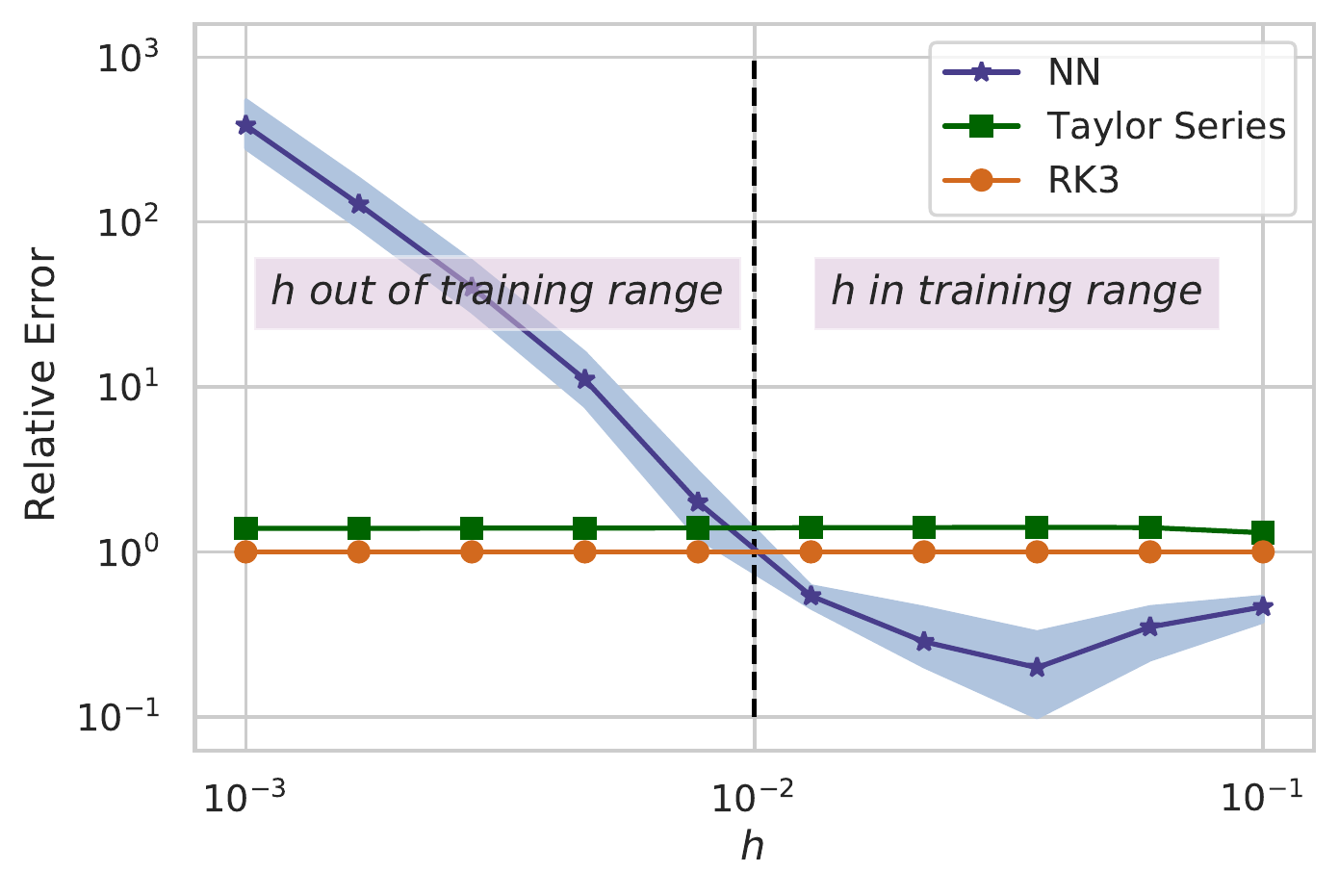}
		\caption{Van der Pol oscillator.}
		\label{fig:vdp}
	\end{subfigure}
	\begin{subfigure}{0.49\textwidth}
		\centering
		\includegraphics[width=1.0\linewidth]{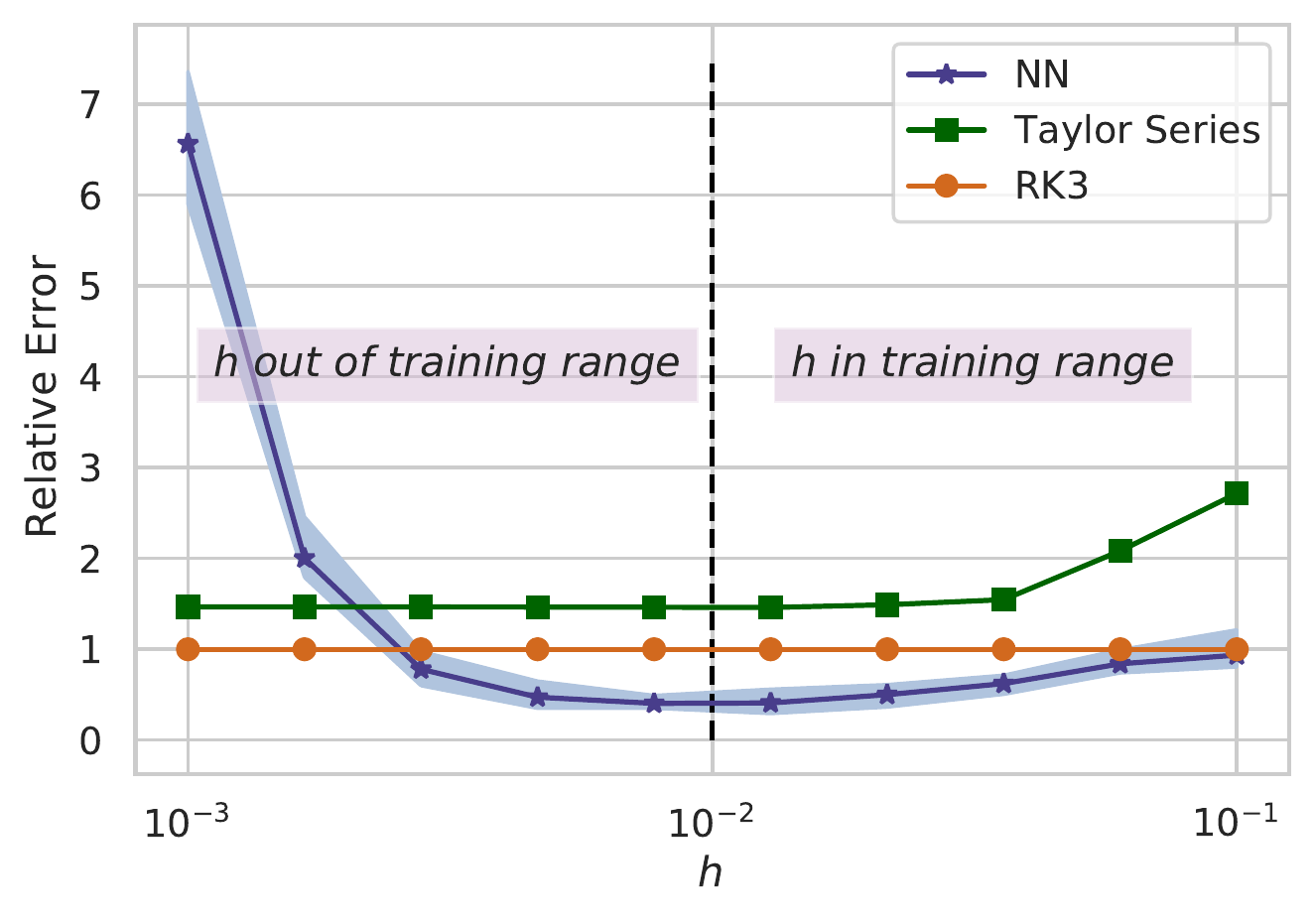}
		\caption{The Brusselator.}
		\label{fig:bru}
	\end{subfigure}
	\caption{Relative error analysis of a three-stage RK-NN integrator compared to RK3. The error of a third-order Taylor Series approximation is also shown.}
	\label{fig:vdp_bru_relative_error}
\end{figure}

\begin{figure}[htb!]
\centering
	\begin{subfigure}{0.49\textwidth}
		\centering
		\includegraphics[width=1.0\linewidth]{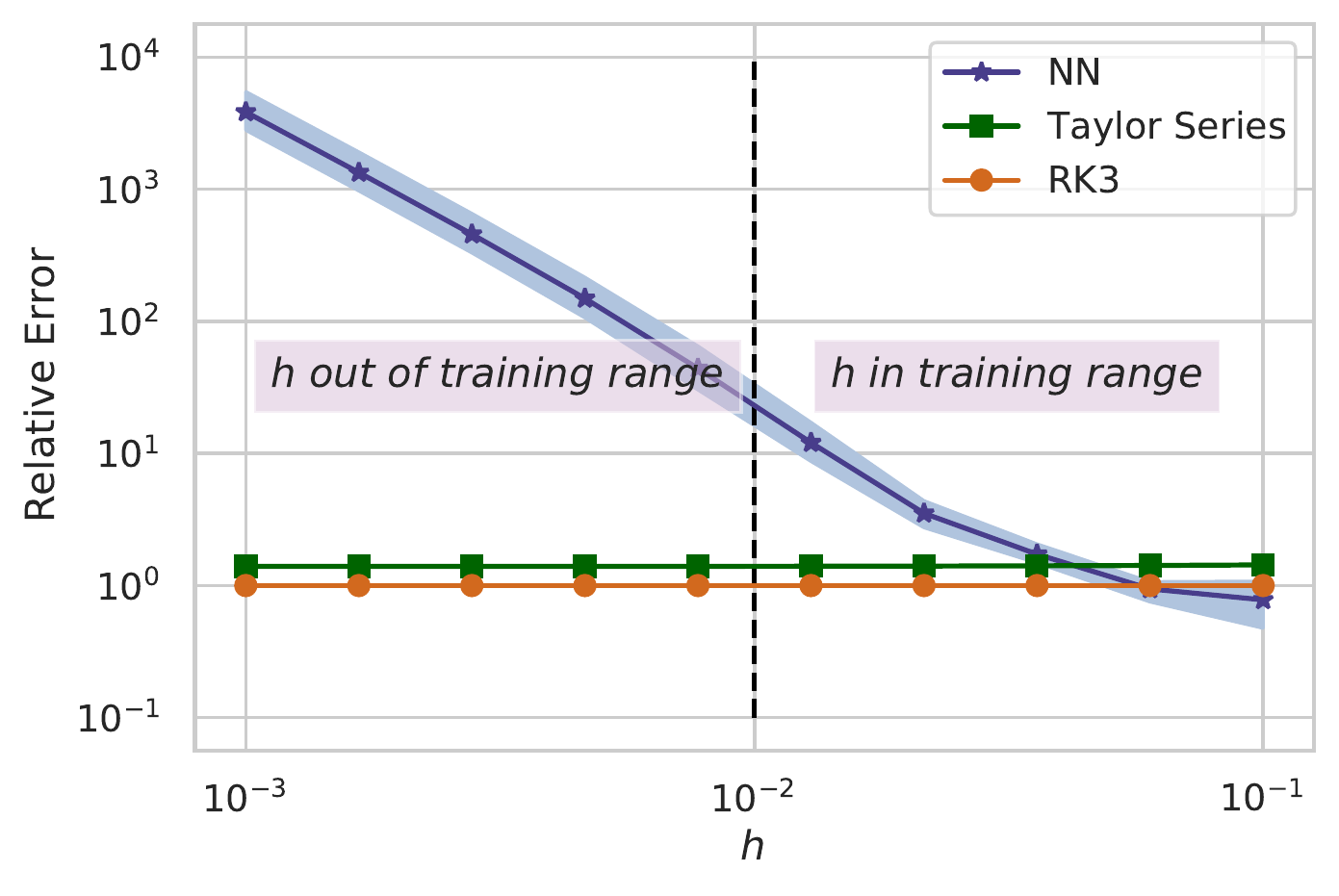}
		\caption{$a \in (0,1)$.}
		\label{vdp_alpha_0_1}
	\end{subfigure}
	\begin{subfigure}{0.49\textwidth}
		\centering
		\includegraphics[width=1.0\linewidth]{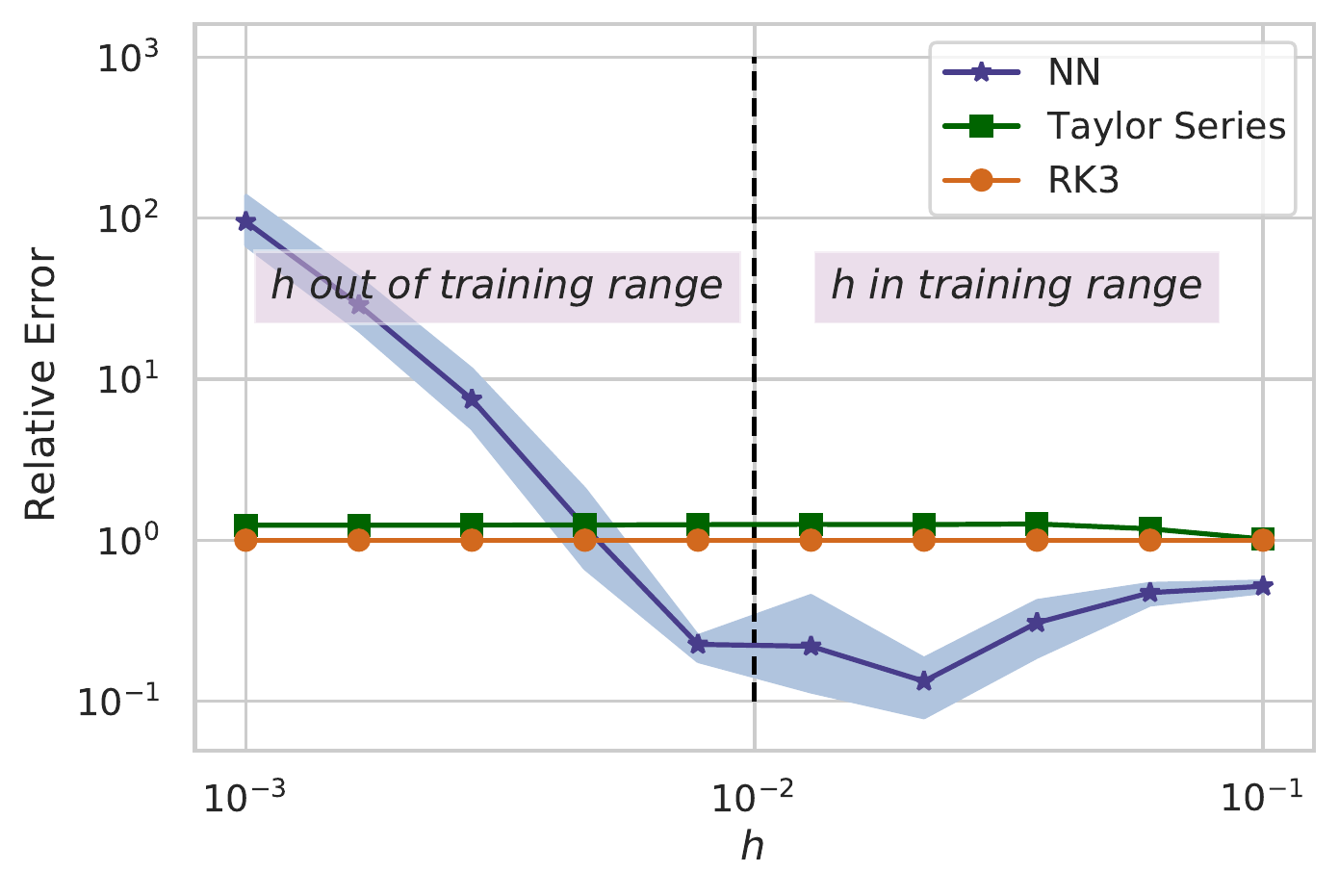}
		\caption{$a \in (2,3)$.}
		\label{vdp_alpha_2_3}
	\end{subfigure}
	\caption{Evaluation on the Van der Pol oscillator with $a$ outside the training range.}
	\label{fig:vdp_alpha_outside}
\end{figure}

\begin{figure}[htb!]
\centering
	\begin{subfigure}{0.49\textwidth}
		\centering
		\includegraphics[width=1.0\linewidth]{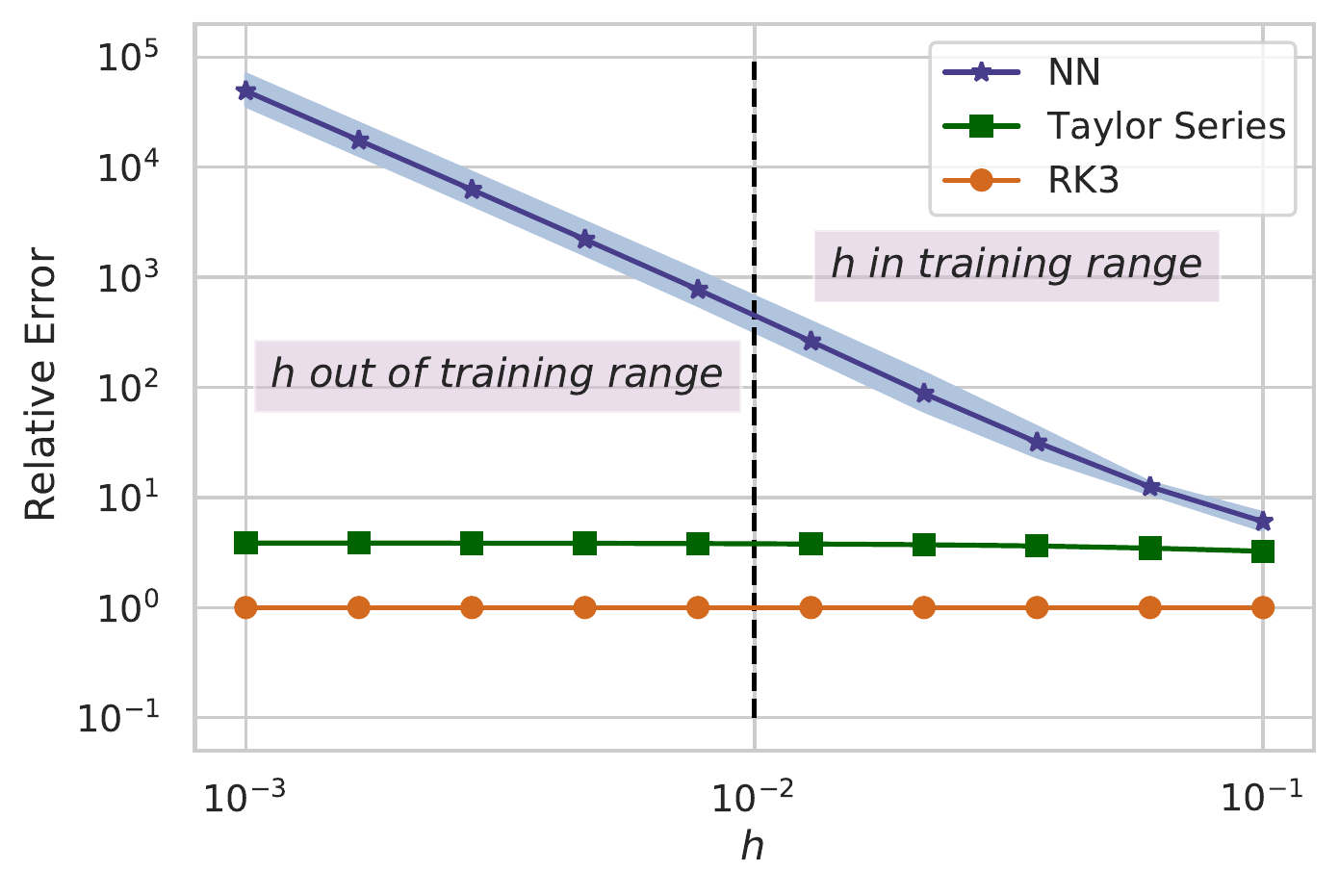}
		\caption{$u_{0} \in (-1,1), v_{0} \in (-1,1)$.}
	\end{subfigure}
	\begin{subfigure}{0.49\textwidth}
		\centering
		\includegraphics[width=1.0\linewidth]{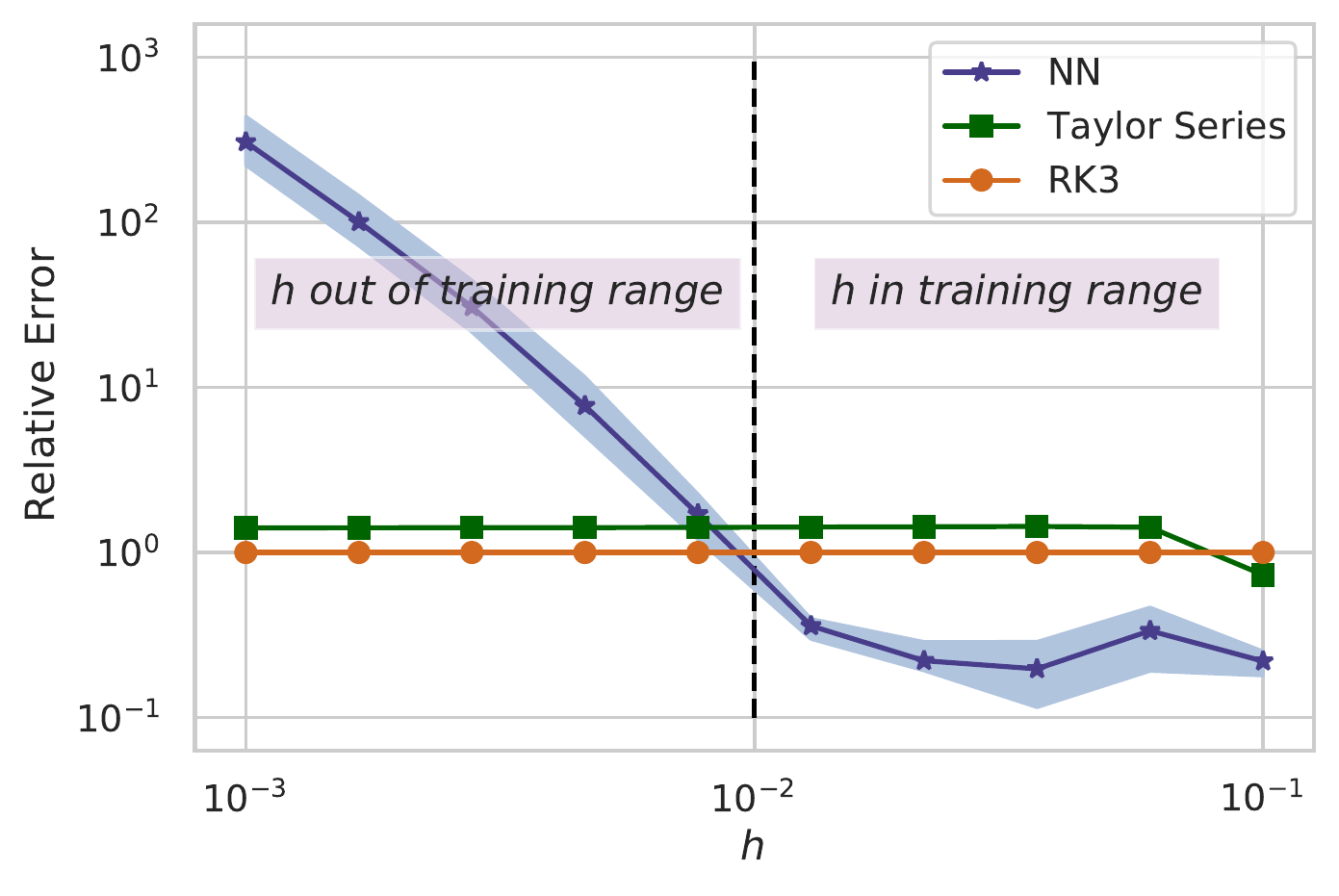}
		\caption{$u_{0} \in (3,4), v_{0} \in (-2,0)$.}
	\end{subfigure}
	\caption{Evaluation on the Van der Pol oscillator with inputs $(u_{0}, v_{0})$ outside the training range.}
	\label{fig:vdp_y0_outside}
\end{figure}

The RK-NN integrator on the Brusselator families is trained with $b \sim U(0.5,2)$, $a=1$ and initial conditions $\{\y_{0}=(u_{0},v_{0}); u_{0} \sim U(1.5, 3), v_{0} \sim U(2, 3)\}$. 
Our method performs favorably inside the training range in \cref{fig:bru}, but has worse performance if the inputs are outside the training range, which is shown in supplementary material \textbf{SM2}. 

To better understand the training process and the role of the two losses (MSE without a scale, Taylor-based regularizer), the curves of losses during the training are shown in \cref{fig:loss_curve}.
Notice that there is a minimum of MSE loss (epoch=300) for the Van der Pol oscillator, and then it increases to a stable value (epoch from 300 to 900).
The performance at that minimum point and subsequent trend are shown in \cref{fig:vdp_mse_min}.
Here, the relative error is smaller compared with that from the final trained integrator when $h$ is around 0.1.
Obviously, as $h$ increases, the prediction will become more inaccurate. Since the MSE error obtained in the training process is the average of the MSE losses at different time steps, the errors obtained when $h$ is large will dominate the average. With this specific parameterization (at the MSE minimum point), the integrator performs well when $h$ is around 0.1, but its accuracy cannot be extended to the entire training range. 
However, the final training results (\cref{fig:vdp}) show that the integrator performs slightly worse under the larger time step, but favorable performance is obtained for all $h$ within the training range. This is enfoced by minimizing the Taylor loss, which ensures the appropriate order of the integrator, despite increasing the MSE.

\begin{figure}[htb!]
\centering
	\begin{subfigure}{0.49\textwidth}
		\centering
		\includegraphics[width=1.0\linewidth]{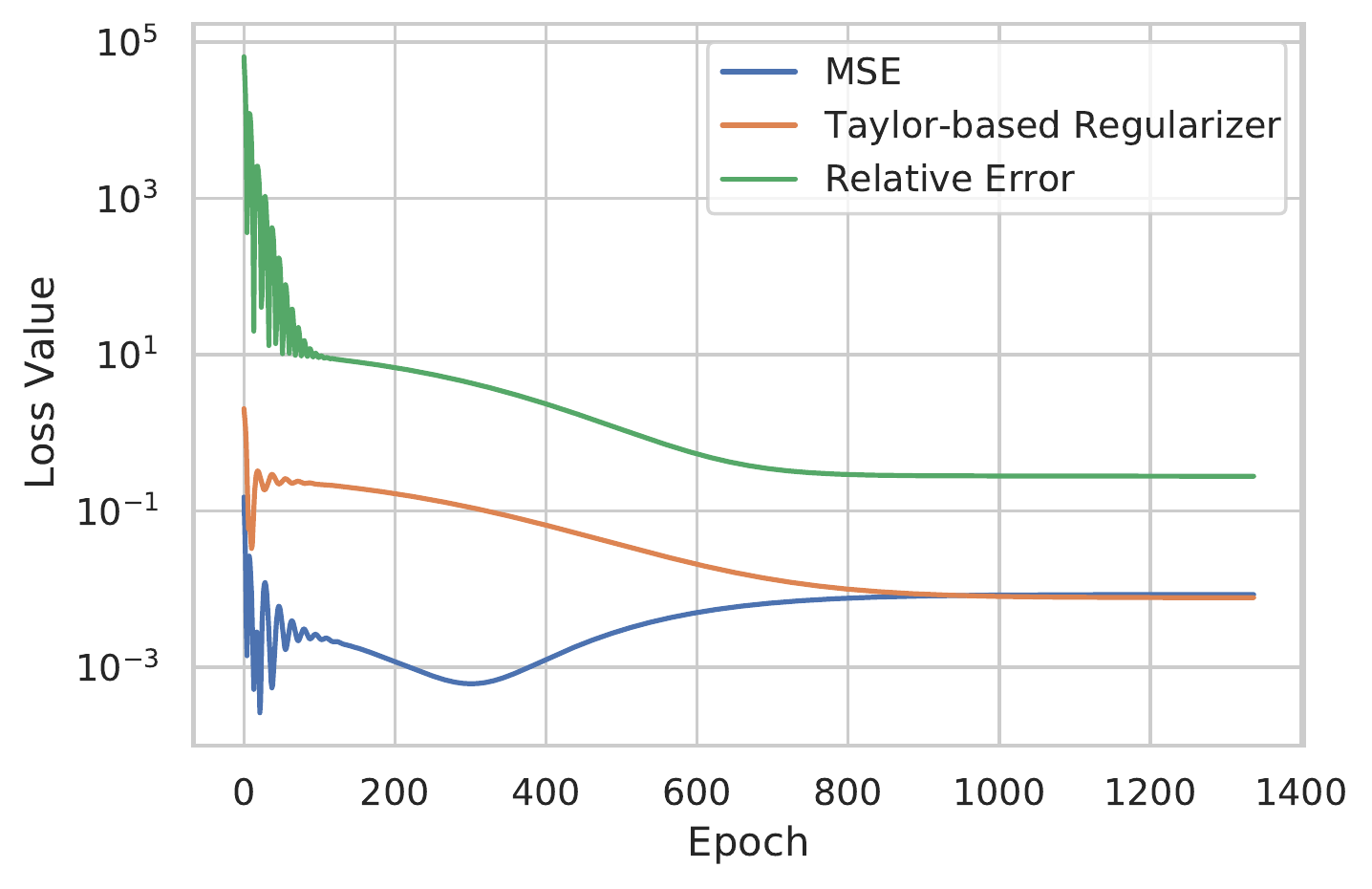}
		\caption{Training loss curve.}
		\label{fig:loss_curve}
	\end{subfigure}
	\begin{subfigure}{0.49\textwidth}
		\centering
		\includegraphics[width=0.95\linewidth]{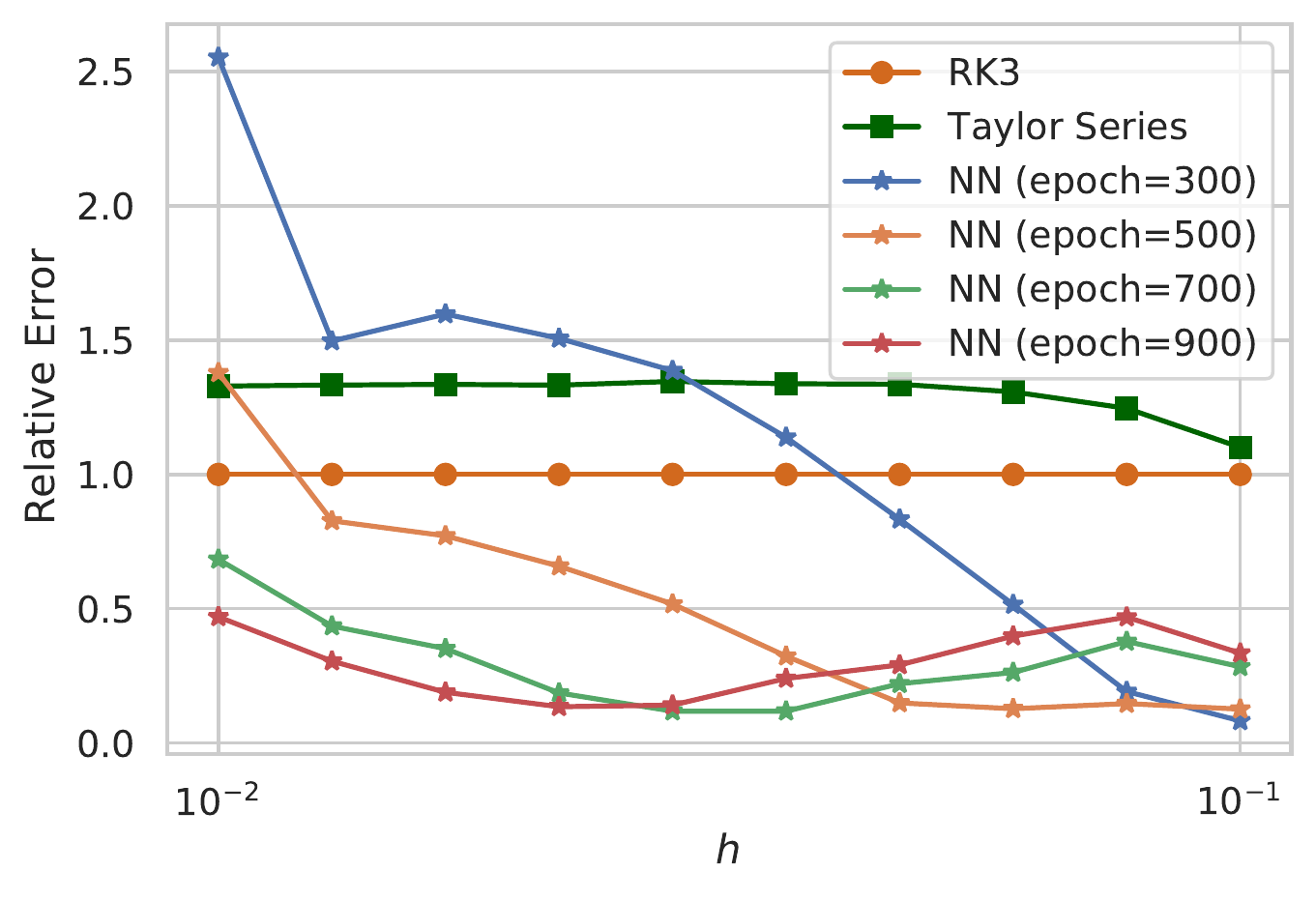}
		\caption{Relative errors at different epochs.}
		\label{fig:vdp_mse_min}
	\end{subfigure}
	\caption{The performance of the Van der Pol oscillator during training is depicted in these two figures. The left hand side shows the loss curves of MSE between the predicted value and true value, Taylor-based regularizer, and the relative error compared with the reference RK method. The right hand side illustrates relative error analysis on the Van der Pol oscillator around the minimum point and later.}
	\label{fig:vdp_loss}
\end{figure}


In summary, we observe that when applied to a problem coming from the family RK-NN is trained on,
the RK-NN can achieve higher accuracy compared with the classical RK methods
within the training range of $h$.
In simple cases, this can be proved theoretically (e.g., \cref{Outperforming}).
In general, we observe that RK-NN finds
``personalized'' RK-like schemes for each problem family with superior performance
(See supplementary material \textbf{SM3} for more quantitative comparisons).
To further illustrate this point, we 
include \cref{table:rk3_paras} that outputs the learned Butcher Tableau values for each problem and compares them with the corresponding coefficients in three generic RK3 methods.
One can observe that in each problem, we obtain specialized values, which illustrates the adaptiveness of our method.

\begin{table}[htb!]
    \caption{Parameter comparisons between three generic RK3 methods and three-stage RK-NN integrators over different task families. We repeatedly trained our RK-NN from multiple random initializations of neural networks and present three arbitrarily selected sets of parameters here, which are denoted by serial numbers in the Van der Pol oscillator and the Brusselator task families, as shown in this table.}
    \label{table:rk3_paras}
    \vspace{20pt}
    \centering
    \begin{tabular}{p{2.5cm}p{1.2cm}p{1.2cm}p{1.2cm}p{1.2cm}p{1.2cm}p{1.2cm}}
        \hline
         & $\theta_{1,1}$ & $\theta_{2,1}$ & $\theta_{2,2}$ & $\theta_{c1}$ &  $\theta_{c2}$  & $\theta_{c3}$ \\
        \hline
        Linear Task  & 0.9682   & 0.1036  & 1.0331                      & 0.5124   & 0.3218  & 0.1657     \\
        Square Task  & 0.9206   & 0.8227  & 1.2625                      & 0.5533   & 0.3704  & 0.0763     \\
        Van der Pol(\#1)  & 0.6880   & 0.3382  & 0.7205                      & 0.4333   & 0.2746  & 0.2920     \\
        Van der Pol(\#2)  & 0.7224   & 0.2493  & 0.6112                      & 0.3718   & 0.3203  & 0.3079     \\
        Van der Pol(\#3)  & 0.7607   & 0.3976  & 0.7357                      & 0.4704   & 0.2749  & 0.2547     \\
        Brusselator(\#1)  & 0.7072   & -0.1350 & 0.7387                      & 0.2462   & 0.4344  & 0.3194     \\
        Brusselator(\#2)  & 0.5459   & -0.1794 & 0.9770                      & 0.2258   & 0.4672  & 0.3070     \\
        Brusselator(\#3)  & 0.6841   & -0.0699 & 0.7215                      & 0.2536   & 0.4178  & 0.3286     \\
        \hline
        RK3(\#1) & 0.6667 & -0.5 & 0.5  & -0.25 & 0.75 & 0.5 \\
        RK3(\#2) & 0.6667 & 0.1667 & 0.5  & 0.25 & 0.25 & 0.5 \\
        RK3(\#3) & 0.5 & -1 & 2  & 0.1667 & 0.6667 & 0.1667 \\
        \hline       
    \end{tabular}
    \label{bs2}
\end{table}


\section{Related Work}
\label{sec:related_work}

In this section, we discuss related work in the literature.
We divide them into four categories, namely learning
solvers by neural networks, learning continuous-time dynamical systems using integrator-embedded parametrization, learning the integrators themselves, and finally meta-learning in general.

First, for learning the operators using ANNs have been proposed in \cite{chen1995universal, raissi2017physics, lu2019deeponet}.
This can be applied in particular to learning anti-derivative operators that can be used for the solution of differential equations. 
We parameterize the solvers using a generalized integrator superstructure, for which we can derive novel loss functions that promote high-order accuracy.

In the second direction,
integrators are useful in computing the solution of dynamical equations and neural networks are effective in approximating general vector fields corresponding to dynamical models \cite{lapedes1988neural, hudson1990nonlinear, krischer1993model}.
Integrator (super)structures like the RK-NNs
have traditionally inspired neural network architectures for effective discovery of vector fields \cite{rico1992discrete}.
Originally, a crucial motivation for the introduction of such methodologies was the approximation of bifurcation diagrams as well as the easy integration of partial physical information (``grey boxes'').
If the right-hand side of equations is fitted, it is feasible to utilize this for tasks other than integration, such as bifurcation calculations using the information included in this equation \cite{siettos2003enabling}.
When the right-hand side of equations is partially known (e.g., $\dd x/\dd t = f(x)+g(x)$ with known $f$ and unknown $g$), the known component can be incorporated into the framework and only the unknown part is left to learn, yielding a natural grey box \cite{rico1994continuous,lovelett2019partial}.
Using fixed RK coefficients from classical Runge-Kutta schemes, the nonlinear vector fields ``$\f$'' can be approximated as neural network modules embedded in the surrounding architecture described by the RK iterations. 
This approach allows for more accurate approximation of the unknown vector field when the trajectory data comes from an underlying continuous-time system, e.g., \cite{rico1993continuous, gonzalez1998identification, wang1998runge}. The resulting approximation error is even related to the specific integration scheme that provides the surrounding architecture \cite{zhu2020inverse}. Consequently, other integrators than RK have also successfully been used to promote the discovery of dynamics, e.g., the trapezoidal rule integrator \cite{rico1995nonlinear, anderson1996comparison}.
We emphasize here that from the learning viewpoint, all these methods have a strong resemblance with the popular residual network (ResNet) \cite{he2016deep}. ODE integrators compute one-step predictions through perturbations of the identity mapping depending on the vector field. If the vector field $f$ is learned from data, this can be considered as a residual network. For example, when we use Euler method to discretize $\dd x/\dd t = f(x)$, one-step prediction from $t_{n}$ to $t_{n+1}=t_{n}+h$ is $x_{n+1} = x_{n} + h f(x_{n})$, which corresponds to the original ResNet structure. Our RK-NN approach can be viewed as a residual network templated on the RK scheme. While ResNet is introduced to address gradient vanishing during training through shortcut connections, RK-NN (when used to learn vector fields defining the dynamics) is designed to recover faithful representation of the continuous time dynamics.
While identifying continuous-time models using the structure of known integrators, \cref{fig:f_nn}, is thus a prominent research area, we consider in this paper the opposite case: given a (family of) driving vector fields, we find the structure of the (optimal) integrator through parametrization as a composite neural network, \cref{fig:integrator_nn}.
Recently, integrator-embedded neural network models have gained a lot of attention as it was demonstrated also that implicit high-order RK methods \cite{raissi2017physicspart2} and arbitrary differential equation solvers can be stacked onto a neural network, extending the applicability of the approach even to tasks beyond the discovery of vector fields.
Related work in this direction include the continuous-time interpretation of residual networks \cite{weinan2017proposal},
its connections to optimal control \cite{Haber2017,Li2017,li2018optimal},
and the now popular neural ODEs \cite{chen2018neuralODE} with applications in generative modelling.

\begin{figure}[htb!]
\centering
	\begin{subfigure}{0.49\textwidth}
		\centering
		\includegraphics[width=1.0\linewidth]{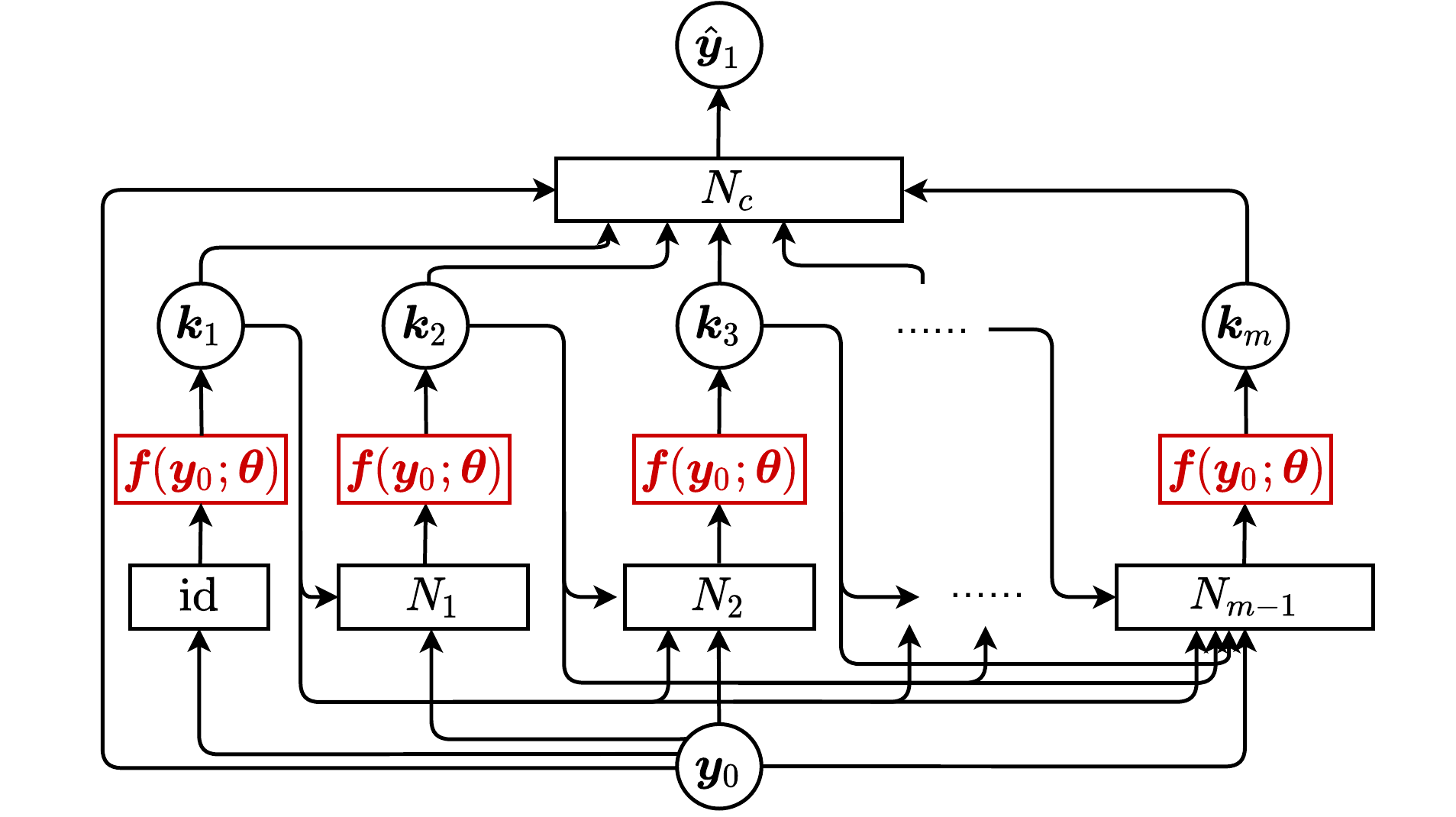}
		\caption{Learning vector field using embedded RK integrator with Neural Net $\boldsymbol{f}(\boldsymbol{y}_0;
\boldsymbol{ \theta})$.}
		\label{fig:f_nn}
	\end{subfigure}
	\begin{subfigure}{0.49\textwidth}
		\centering
		\includegraphics[width=1.0\linewidth]{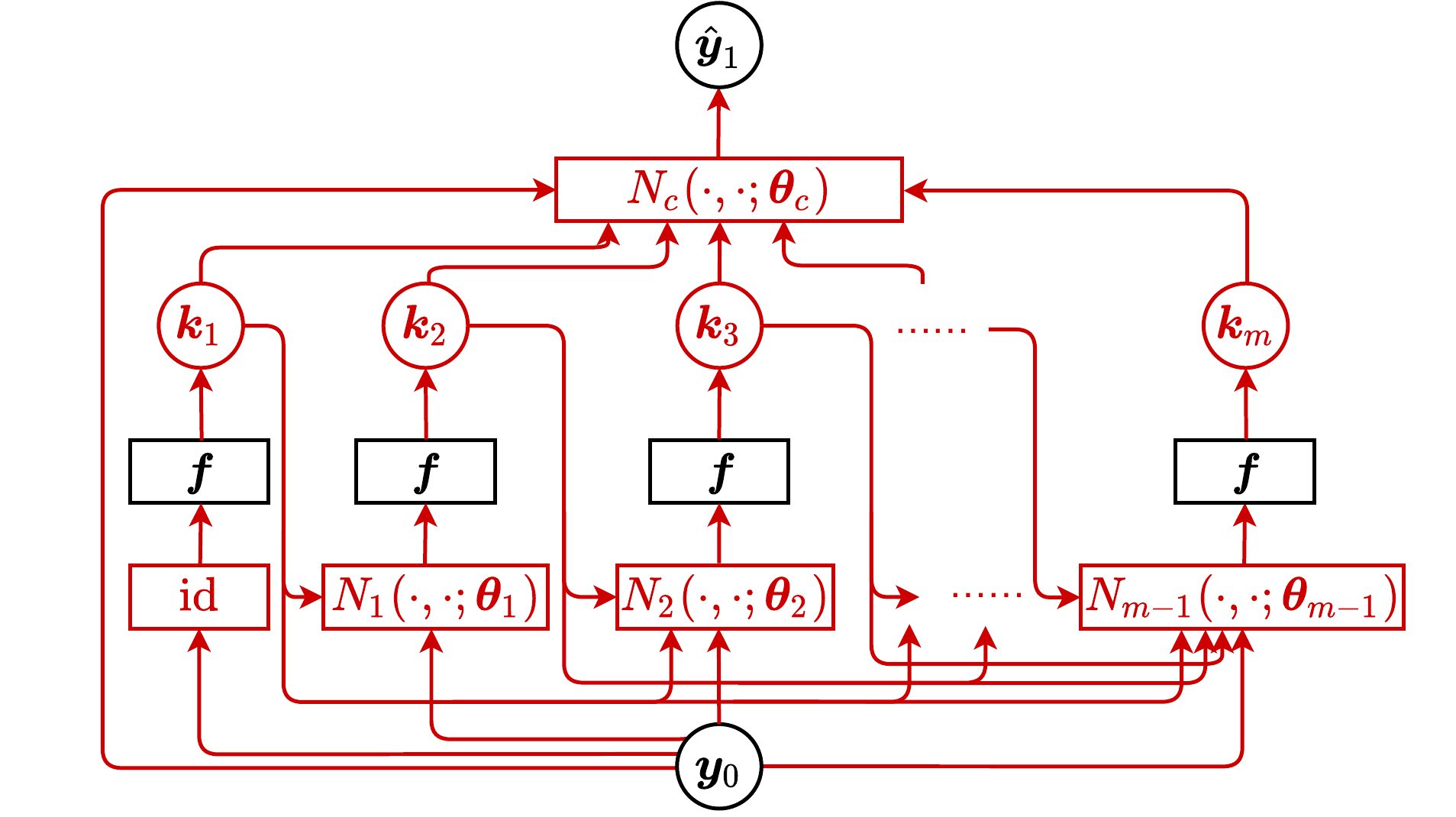}
		\caption{Learning the integrator parameterized by general form of RK method.}
		\label{fig:integrator_nn}
	\end{subfigure}
	\caption{Schematic representation of the implementation of RK method. Highlight is approximated by neural networks. The vector field is trainable in \cref{fig:f_nn} and the integrator is built by neural networks in \cref{fig:integrator_nn}. Here $\hat{\y}_{1}$ is the prediction by RK integrator with initial condition $\y_{0}$. $N_{c}$ and each $N_{i}$ describe weighted summation.}
	\label{fig:comparison}
\end{figure}

In this reverse direction of designing integrator structures, there are a number of previous related studies have addressed this problem setting \cite{Fehlberg-1964,Fehlberg-1966,Fehlberg-1969,fehlberg1970klassische,fehlberg1975klassische,suli2003introduction}.
In these works, the authors numerically compute RK coefficients from analytically derived constraints.
With the combination of neural networks,
the parameters in integrator structures can be learned to meet specific requirements.
To learn the coefficients in RK4, one constructs the loss function by the scaled distance between true and approximated value \cite{tsitouras2002neural}.
\cite{anastassi2014constructing} and \cite{dehghanpour2015ann} are respectively concerned with the RK2 and RK3 method and augment the loss function to penalize deviation of the coefficients from the equations that are required for RK methods.
However, the present method differs in a number of ways.
First, we introduce a new loss term, based on Taylor series, which can automatically discover high-order integrators, without the need to manually derive weight constraints, as is done in \cite{iserles2009first}. This makes our method easy to implement and scalable to high orders of accuracy.
Second, having trajectory data from the true solution is not compulsory in our method, even though previous work always relies on true labels to calculate the error, since the Taylor series expansion is a local property and is entirely determined by the vector field $\f$ (See \cref{subsec:loss}).

Lastly, from the machine learning viewpoint, our approach aims to learn solvers that perform well on a target family of tasks, rather than one fixed task.
Our RK-NN method is a form of multi-task learning in a broader sense.
Multi-task learning focuses on learning several tasks simultaneously so that the knowledge contained in one task can be applied to other tasks \cite{crawshaw2020multi,zhang2017survey}.
However, it is worth noting that there are some important differences to the neural network architectures designed for multi-task learning \cite{crawshaw2020multi}.
The latter aims to find shared neural network feature extractors relevant
for the multiple pre-defined tasks,
each of which is accompanied by specific task-dependent final layers.
While we also consider multiple tasks in our architecture,
most importantly, in our setting the algorithm must generalize to new unseen tasks in the task family.
In terms of the generalization,
our problem setup is thus related to specific branches of meta-learning \cite{Thrun1998} that are concerned with finding hyperparameters, neural networks architectures, parameter initializations etc. suited for a family of tasks \cite{Hospedales2020metalearning,finn2017model,nichol2018first,rajeswaran2019meta}.
It is a worthwhile endeavor for future research to investigate how the extensive toolbox of meta-learning methods available for neural network architectures can be leveraged to integrate even further adaptation to problem instances into our approach.

\section{Summary and Outlook}
\label{sec:discussion}


In this paper, we study how to effectively combine machine learning and numerical
analysis ideas to arrive at efficient and adapted algorithms.
As a case study, we developed a method to automatically learn high-order integrators for specified ODE families, based on the RK algorithmic superstructure.
A key idea is the definition of a RK-like neural network architectural superstructure (RK-NN), together with a Taylor series based regularizer that ensures high order accuracy and adaptivity to the problem class.
Instead of computing the Butcher tableau, we focus on the performance of our model under multiple tasks following some task distribution. Based on that, we can sample and train RK-NN by minimizing the sum of the scaled least-squares error and the regularizer. 
In the average sense under this distribution, the method has superior performance to the classical RK method because it can exploit the structure that may be present.
We apply this method to various examples, including the Van der Pol oscillator and the Brusselator, where we demonstrate that RK-NN brings lower global truncation error than the classical RK method.
Overall, this represents a basic step towards a systematic investigation of learning-based approaches towards numerical algorithm design and adaptation.

Here, we considered minimizing the global errors among distinct integrators with respect to some fixed time step.
However, prior work has shown that numerical computation can be improved if we adapt the step size during the evaluation.
The order of accuracy of a numerical integration scheme is a convergence statement for vanishing step size.
In practice, and with finite step size, the actual errors of the numerical solution compared to the true solution can vary widely.
A typical approach to control this is to use a higher-order method in each iteration to estimate the local error. The step size will be reduced if this  estimated error is above a user-defined tolerance.
Seminal work on this idea was done by Romberg~\cite{Romberg-1955} in the context of integration.
Considering the explicit forward Euler method as an example, the adaptive step size method uses half of the current step size as a more accurate version. If the error estimate is below the tolerance, we consider the step successful and continue with the next step. Else, we adapt the step size through a specific equation.
A particular challenge for adaptive step size control is the increased computational cost due to the ``second" evaluation of the method in each iteration. To minimize the number of function evaluations, E. Fehlberg developes multiple versions of the Runge-Kutta type integrators and minimizes the coefficients in the scheme while simultaneously minimizing the error of the fourth-order scheme \cite{Fehlberg-1964,Fehlberg-1966,Fehlberg-1969,fehlberg1970klassische,fehlberg1975klassische}, leading to ``Runge-Kutta-Fehlberg''.
Dormand and Prince later develop a scheme that optimizes the coefficients for a 5th-order method. They also use the first-same-as-last (FSAL) property to reduce the number of function evaluations necessary for both a fourth and fifth order accurate scheme~\cite{dormandprince-1977,dormand-1980b}.
In particular, their method of combining six function evaluations to obtain results for a fourth- and fifth-order accurate RK scheme simultaneously has been implemented in the \textsc{ODE45} solver of MATLAB.
Inspired by these studies, we can design an algorithm based on the approach introduced in this paper to automatically learn the adaptive step size for improved robustness in the future.

Throughout this paper we have considered a ``fixed superstructure": RK-NNs with a fixed recurrent architecture, thus restricting the search space to RK integrators of a fixed stage and with fixed zero entries in the Butcher tableau.
In the future, we envision that more choices will be left to the optimizer, i.e., the number of RK stages, together with all entries in the corresponding Butcher tableau will be learned as part of a meta-learning procedure.
This in general includes implicit methods, as explored in \cite{anderson1996comparison}.
Such additional degrees of freedom would define a problem that encompasses the general family of RK-based algorithms. We have started to work on more general Krylov-inspired recurrent NN architectures.
In order to find an optimal algorithm within such a superstructure, {\it the neural network architecture and its parameters} need to be optimized jointly.
The use of global, mixed-integer optimization over superstructures in order to discover optimal algorithms has been advocated (and illustrated) in \cite{Mitsos2018algorithms}, where the meta-learning problem was formulated as an optimal control problem.
It is interesting to summarize the experience of these authors with their approach, which they called ``optimal algorithm generation" rather than ``meta-learning": Depending on the subclass of problems over which they optimized, they sometimes found standard algorithms (that are documented in textbooks); at other times they found algorithms unknown (to them), that, upon literature search, had been previously discovered; and upon occasion, they found algorithms for the problem subclass that ``made perfect sense", and that somebody could have discovered, but apparently had not yet been discovered and documented.
A major conclusion was that the optimal algorithm was extremely sensitive to the fine scale details of the particular problem subclass, s.t. algorithms both previously unknown and also general could not be found.
Our argument here is that this ``generality weakness" is in effect a ``personalized strength": that we will maybe have integrators (and more generally algorithms, and sometimes hardware computers, like D.E.Shaw's ANTON \cite{shaw2008anton}) tuned to a particular class of equations, with particular types/ranges of initial/boundary conditions, over particular scales. The recent paper by Brenner and co-workers \cite{bar2019learning} on ML-discovered, problem-dependent PDE discretizations also moves in the same overall direction. 
The more times a problem needs to be solved repeatedly, the more the effort for discovering a ``personalized" optimal algorithm is justified.

In the context of machine learning, one particular type of meta-learning, called neural architecture search (NAS) refers to a set of techniques for finding the optimal network architecture for a given task \cite{Hospedales2020metalearning}.
Beyond the search space, which is herein defined by the superstructure, NAS methods are categorized by their search strategy and their performance estimation strategy \cite{Elsken2019neural}. Current research is focused on search strategies based on reinforcement learning \cite{Zoph2017neural} and Bayesian optimization (e.g., \cite{Kandasamy2018neural,Liu2018progressive}), while evolutionary algorithms have been employed in NAS for decades (e.g., \cite{Angeline1994evolutionary}).
Pruning, i.e., deletion of neural network parameters based on some importance metric, can be interpreted as a widely used search strategy for NAS when the aim is to discover sparse model representations. Early pruning methods were based on Taylor series approximations of the sensitivity of the loss function with respect to network parameters \cite{LeCun1990optimal,Hassibi1993optimal}. For deep networks, pruning based on weight magnitude is an established approach (e.g., \cite{Han2016deeppruning,See2016compressionNLP}), but recent work also considers information criteria such as synaptic saliency to achieve the highest possible model sparsity \cite{Tanaka2020synaptic}.
To reduce the complexity of the search space NAS methods can further exploit modularity within the superstructure \cite{Zoph2018learning,Boecking2020modular,Negrinho2019modularprogrammable}. For instance, for the RK-NN, it is particularly useful to think of all layers involving a function call and all skip connections between these layers as prunable units.

Beyond NAS, meta-learning on a superstructure would facilitate optimizing algorithms with respect to a metric that depicts an optimal trade-off between the complexity and the accuracy of each algorithmic step. A prime example for such a metric is the wall-clock performance of the algorithm for achieving some pre-defined stopping criterion,
which is determined by the number of iterations and the cost per iteration
(see \cite{Mitsos2018algorithms})
Since the cost per iteration depends on the number and type of mathematical operations it involves, e.g., function evaluations, matrix-vector products or computations of a Jacobian, the corresponding superstructure optimization problem constitutes a mixed-integer nonlinear program (MINLP).
The NAS methods above are capable of finding approximate solutions to this problem due to their heuristic nature and are thus widely used.
In contrast, integer programming (IP) algorithms search the space of possible architectures in a more rigorous manner than NAS methods by constructing relaxations and introducing cutting planes.
However, IP algorithms exhibit prohibitive cost for larger networks and are thus not applied to training. One possible exception is \cite{ElAraby2020identifying}, in which the authors circumvent the full complexity of the discrete problem by evaluating neuron importance with a binary value and training a network to minimize the number of important neurons to achieve sparsity.
For a smaller network such as the one we used in this work, discrete optimization appears tractable. However, we point out that the theory of IP algorithms requires global solution of all relaxed subproblems with each subproblem corresponding to a neural network training problem. While training may lead to global optima empirically \cite{Goodfellow2014qualitatively}, and under some circumstances provably \cite{Haeffele2017global}, this observation does not hold for all neural networks \cite{Frankle2020revisiting}.
Therefore, to {\it guarantee} finding the best solution, deterministic global optimization is necessary; \cite{Mitsos2018algorithms} illustrate this. Neural network architectural superstructures allow for more flexible function bases and may be optimized deterministically in similar fashion by the algorithms for superstructure optimization given in \cite{Grossmann.2002} and \cite{Mencarelli2020superstructure}.
We do not envision (due to the tremendous computational difficulty of the problem) that mixed integer global optimization over superstructures will soon become the standard tool for optimal NN architecture generation for a ``personalized" computational task over a class of problems of interest.
More practical tools, e.g., exploiting modular structure, and informed pruning, will certainly carry the day in the foreseeable future.
Yet the superstructure formulation of the optimization problem, and crucially the construction of an intelligent and flexible superstructure, informed by calculus and traditional numerical analysis, appears to us a truly worthy research task in this meta-learning quest for personalized algorithm generation.

\bibliographystyle{siamplain}
\bibliography{library}

\newpage
\centerline{\textrm{\textbf{SUPPLEMENTARY MATERIALS}}}

We introduce some additional results in these supplementary materials.
In \textbf{SM1},
a four-stage RK-NN integrator is trained to sixth-order accuracy over a range of $h$ in a specific task family, which breaks the order barrier of the classic RK method.
\textbf{SM2} illustrates worse performance on the Brusselator is obtained when initial condition $\y_{0}$ or the equation parameter $b$ is out of the training range. We also evaluate both the Van der Pol oscillator and the Brusselator to see what happens when time step $h$ is out of training range in 
\textbf{SM3}.
In 
\textbf{SM4}, RK-NN is compared to three classic RK3 methods with different parameterization.
Finally, we show the training complexity of a linear task by plotting the wall-clock time required for one epoch in
\textbf{SM5}.

\section*{SM1. Additional outperforming experiment results}
\label{RK4_order_6}
We apply the same approach (Alg.1) to train a superior four-stage RK-NN with $\alpha=6$ and $m=4$.
It is well known that one requires a seven-stage RK method to obtain sixth-order accuracy for generic ODEs. 
~\Cref{fig:rk4_order6} shows that we can obtain a sixth-order integrator with four-stage RK-NN for square task family \cref{square_target} (detailed definition is illustrated in section 4 of the main article).
\begin{align}
\label{square_target}
\begin{split}
    \Fcal &= 
    \{ \y \mapsto -a\y^2 \mid a > 0\} \times \{ \R \},\\
    \mu &= \mathrm{Distribution}(\{\y \mapsto -a\y^2; a \sim U(0.1, 0.5)\}) \times U(1, 3).
\end{split}
\end{align}

\begin{figure}[htb!]
	\centering	
	\includegraphics[width=0.6\textwidth]{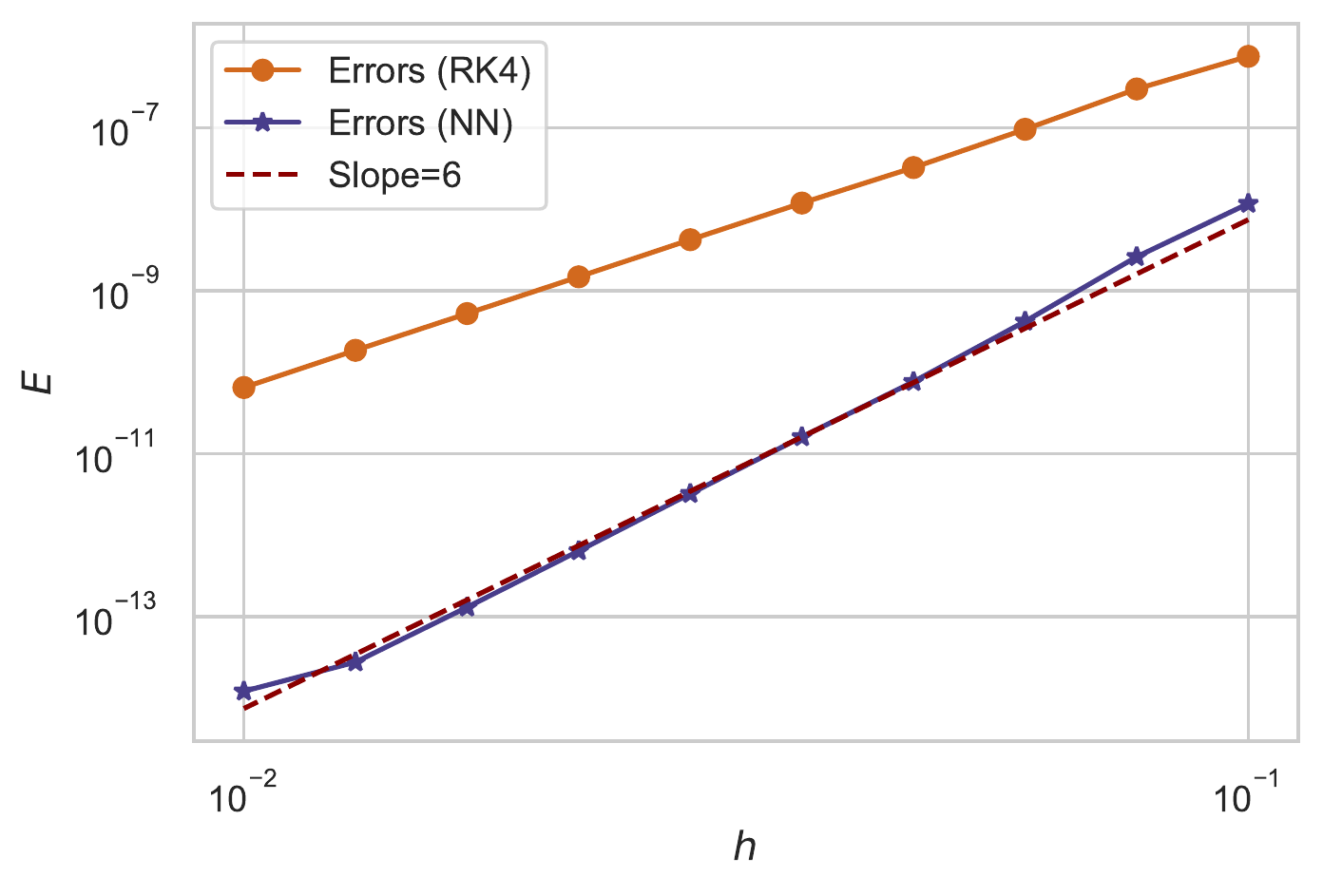}
	\caption{Error analysis on \textbf{square task} family, after training and testing on $h\in (0.01,0.1)$, using \textbf{four-stage} RK-NN integrator with \textbf{sixth-order} Taylor-based loss as regularizer.}
	\label{fig:rk4_order6}
\end{figure}

\section*{SM2. Error analysis on the Brusselator with initial conditions or parameters outside training range}
\label{bru_outside}
Training RK-NN integrator on the Brusselator families with $b \sim U(0.5,2)$ and $a=1$, we use $\{\y_{0}=(u_{0},v_{0}); u_{0} \sim U(1.5, 3), v_{0} \sim U(2, 3)\}$ as initial condition.
We test on the initial condition $u_{0} \in (0.5,1), v_{0} \in (1,2)$ or parameter values $b \in (3.5, 4)$.
Evaluation results are shown in \cref{fig:bru_outside}.

\begin{figure}[htb!]
\centering
	\begin{subfigure}{0.49\textwidth}
		\centering
		\includegraphics[width=1.0\linewidth]{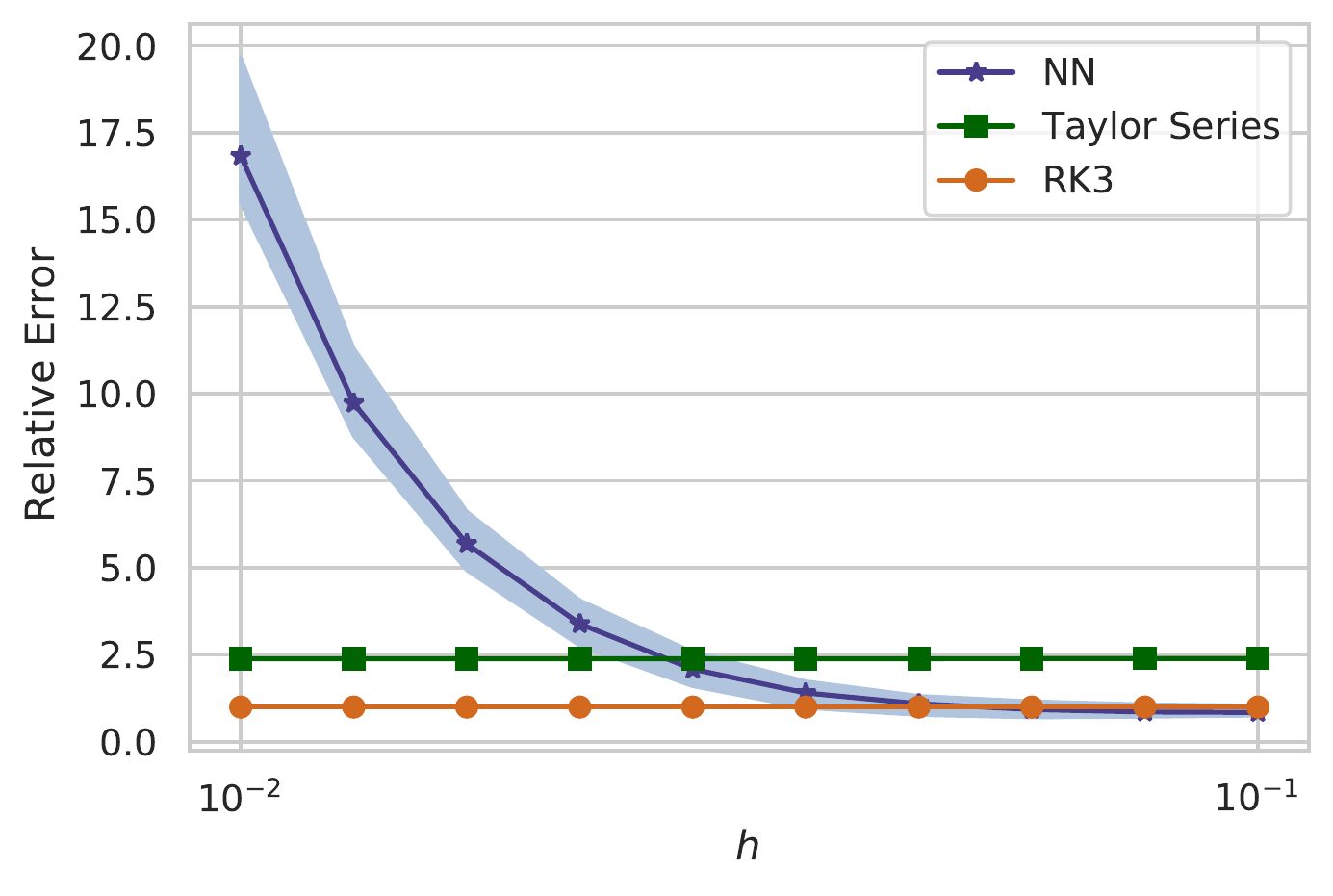}
		\caption{$u_{0} \in (0.5,1), v_{0} \in (1,2)$.}
	\end{subfigure}
	\begin{subfigure}{0.49\textwidth}
		\centering
		\includegraphics[width=1.0\linewidth]{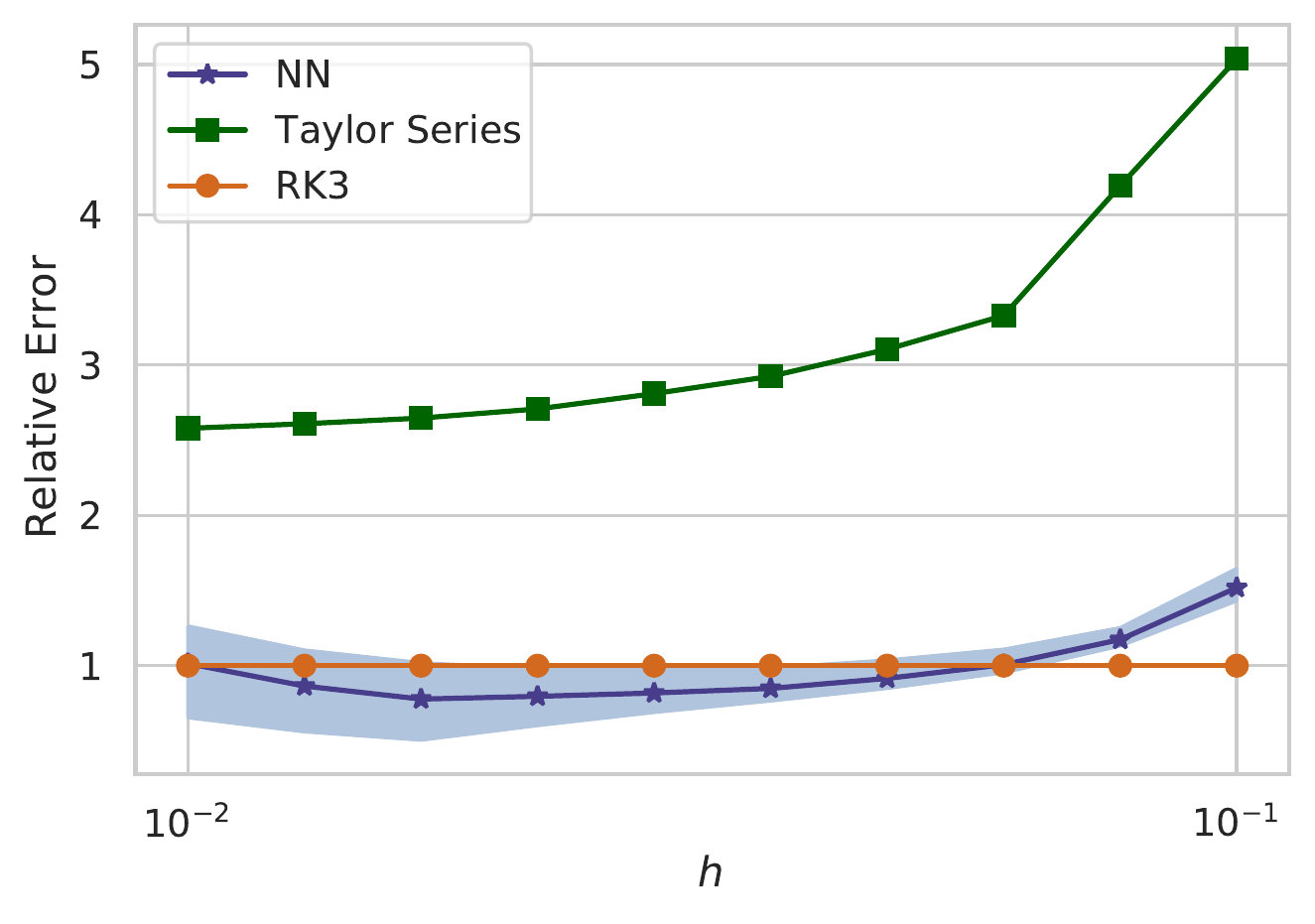}
		\caption{$b \in (3.5, 4)$.}
	\end{subfigure}
	\caption{Evaluation on the Brusselator by using inputs outside the training range.}
	\label{fig:bru_outside}
\end{figure}

\section*{SM3. Evaluation on the Van der Pol oscillator and the Brusselator while time step $h$ out of training range}
\label{vdp_bru_h_outside}
In the previous experiments, we train and test on $h \in (0.01, 0.1)$.
We test the trained integrator on a larger time step range, such as $h \in (0.001, 0.1)$, to study its generalization in terms of time steps.
\Cref{fig:h_small} and \cref{fig:h_small_ordercheck} show that the  performance worsens when $h$ goes to a smaller value.
This demonstrates that the learned integrators
are also adapted to the range of step sizes during training.
Note that for most applications, the current range of step sizes gives sufficiently small errors, thus this is not a major issue of the method.
\begin{figure}[htb!]
\centering
	\begin{subfigure}{0.49\textwidth}
		\centering
		\includegraphics[width=1.0\linewidth]{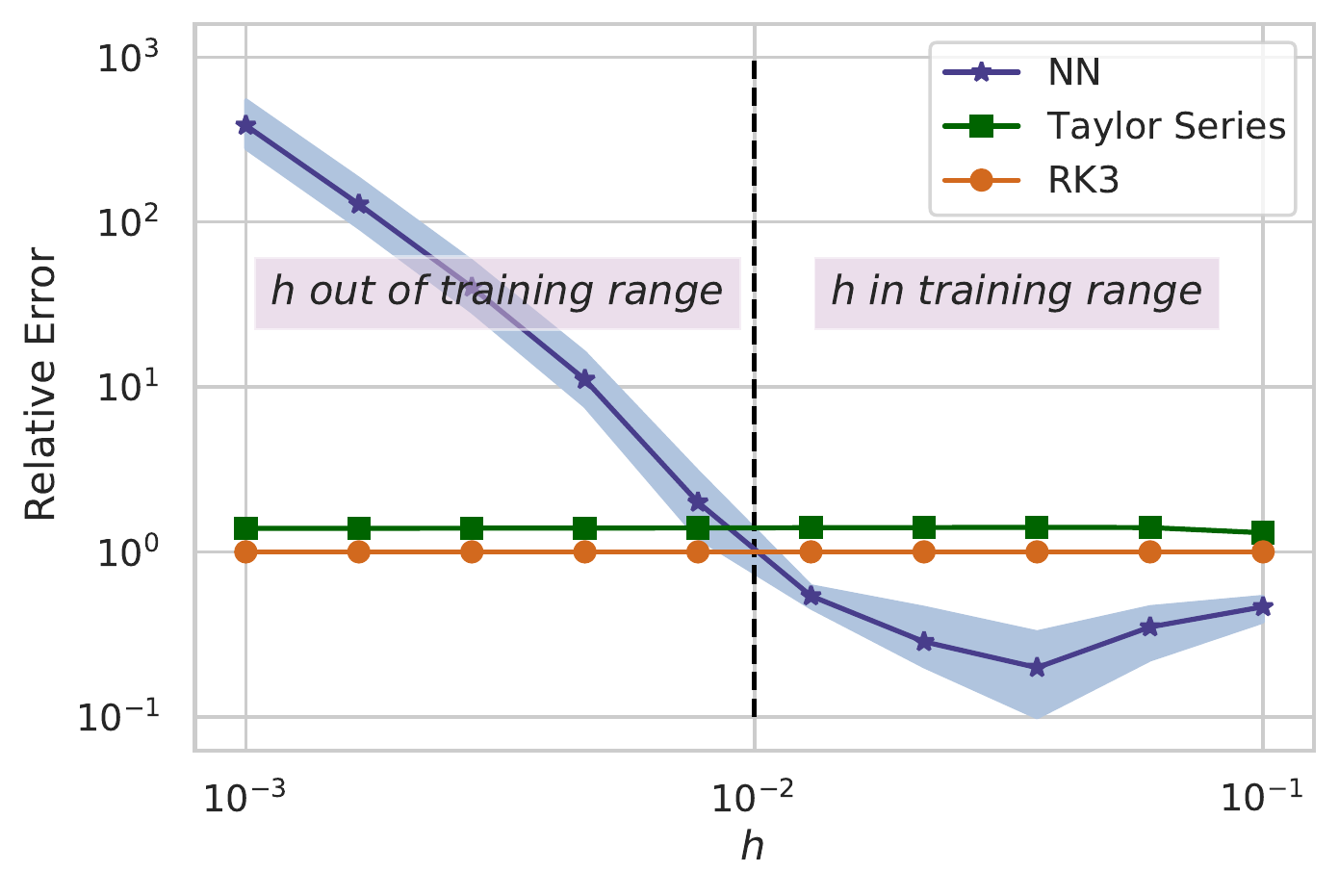}
		\caption{Van der Pol oscillator.}
	\end{subfigure}
	\begin{subfigure}{0.49\textwidth}
		\centering
		\includegraphics[width=1.0\linewidth]{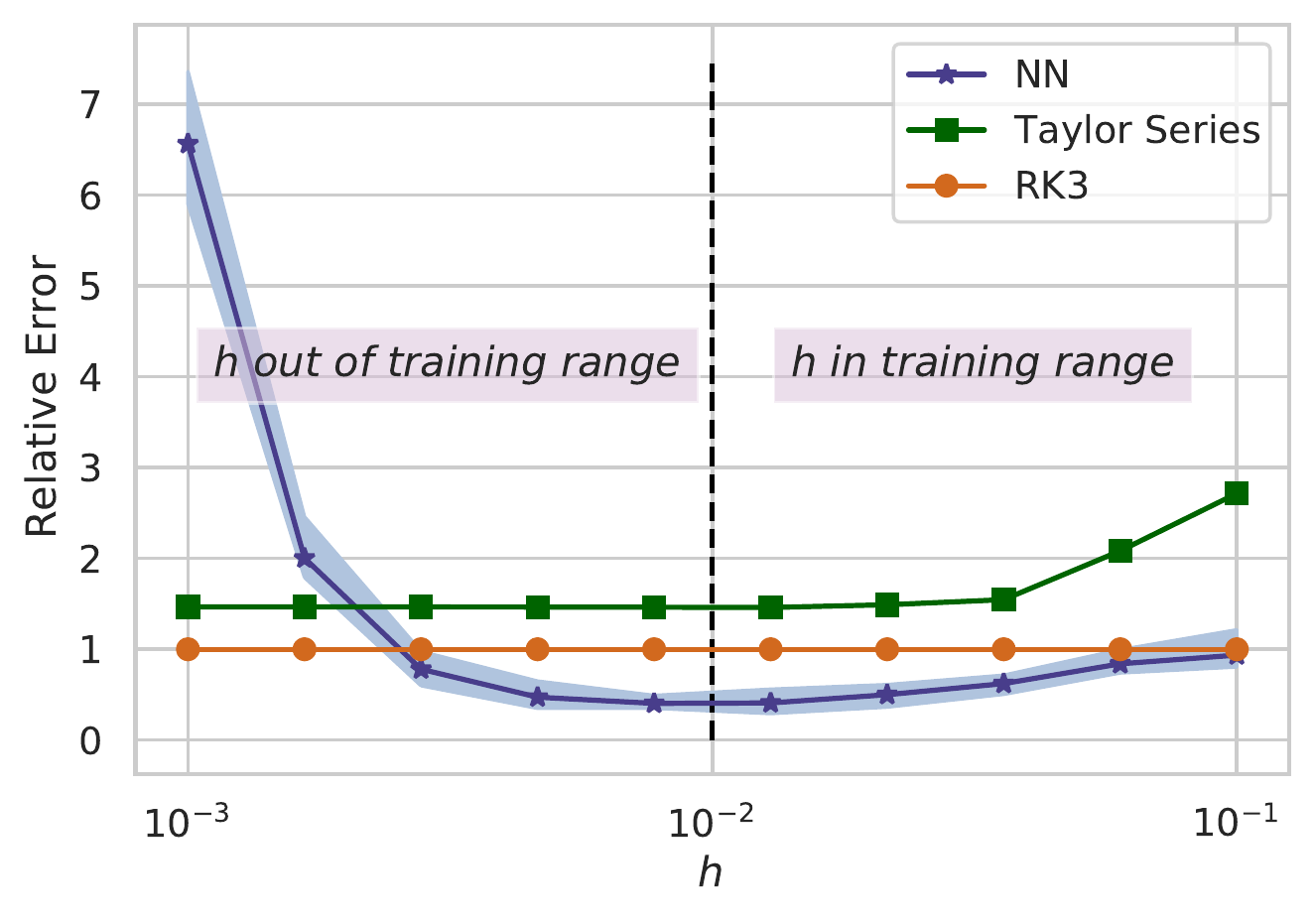}
		\caption{The Brusselator.}
	\end{subfigure}
	\caption{Evaluation on the relative error when $h$ out of training range.}
	\label{fig:h_small}
\end{figure}

\begin{figure}[htb!]
\centering
	\begin{subfigure}{0.49\textwidth}
		\centering
		\includegraphics[width=1.0\linewidth]{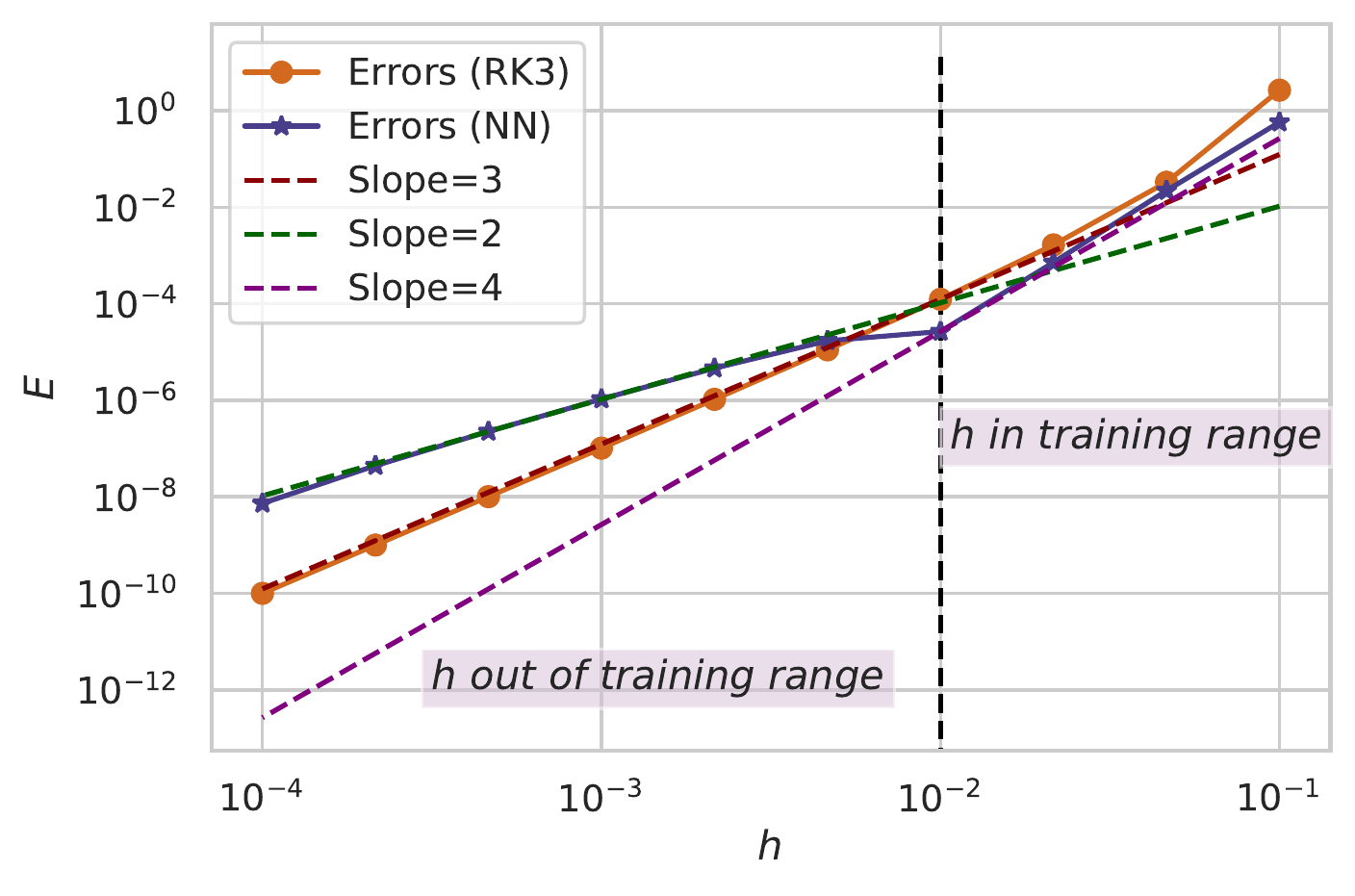}
		\caption{Van der Pol oscillator.}
	\end{subfigure}
	\begin{subfigure}{0.49\textwidth}
		\centering
		\includegraphics[width=1.0\linewidth]{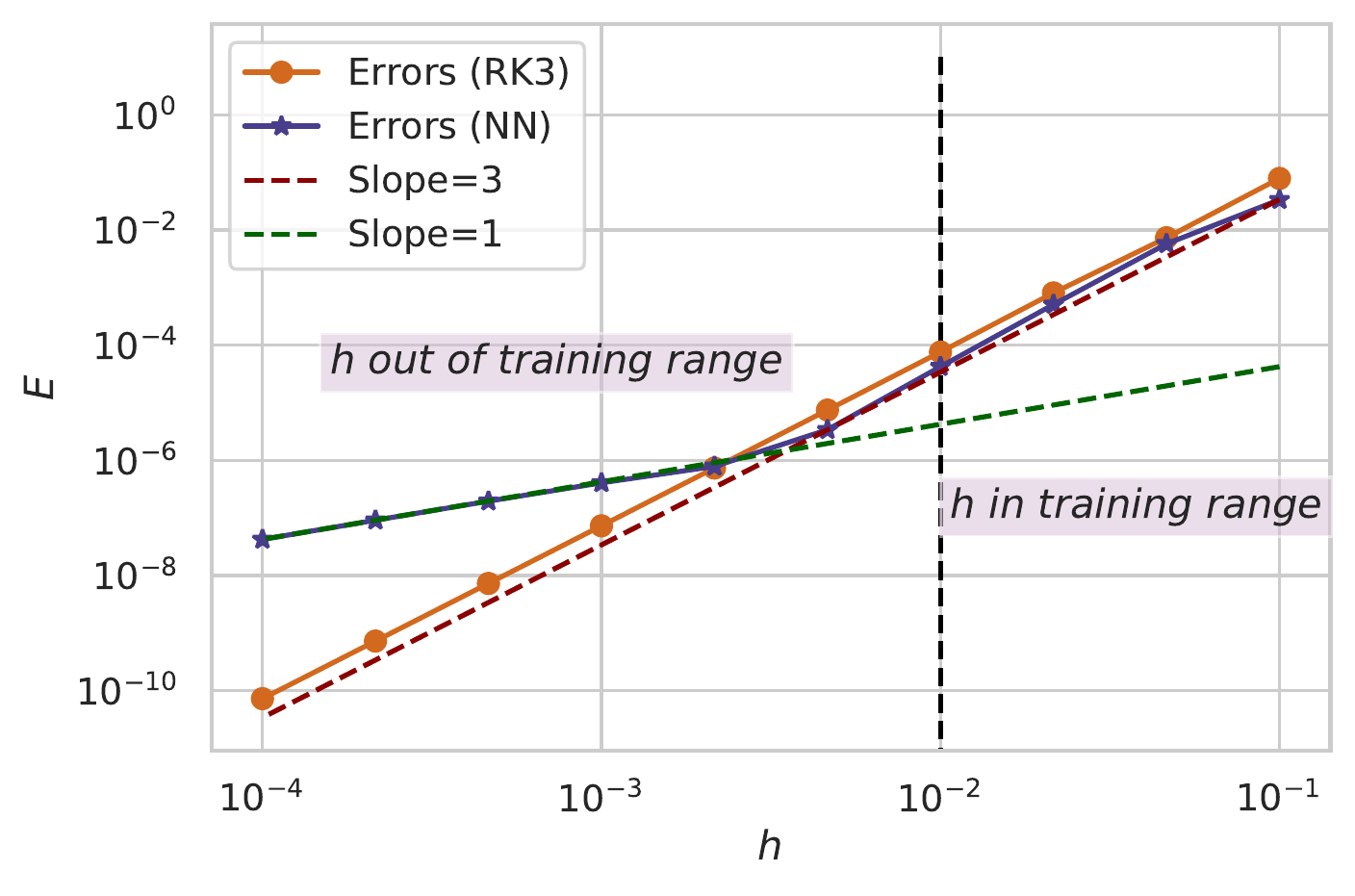}
		\caption{The Brusselator.}
	\end{subfigure}
	\caption{Evaluation on the global error when $h$ out of training range.}
	\label{fig:h_small_ordercheck}
\end{figure}

\section*{SM4. Comparison with three different traditional RK3 methods}
\label{diff_rk3}

In this paper, our goal is to find an approximate Butcher tableau for some given problems and range of time step $h$.
To figure out whether our method is better, we decide to calculate results by RK3 methods with different parameters on our examples.
Recall the formulation of RK3:
\begin{align}
\begin{split}
    \kk_{1} = h f(\boldsymbol{y}_{n}), \quad
		\kk_{2} &= h f(\boldsymbol{y_{n}} + \theta_{1,1} \kk_{1}), \quad
		\kk_{3} = h f(\boldsymbol{y_{n}} + \theta_{2,1} \kk_{1} + \theta_{2,2} \kk_{2}),\\
		\boldsymbol{y}_{n+1} &= \boldsymbol{y}_{n} + \theta_{c1}\kk_{1} + \theta_{c2} \kk_{2} + \theta_{c3}\kk_{3}.
\end{split}
\end{align}
Here is shown three different parameterized RK3:
\begin{itemize}
    \item RK3 (\#1) $\theta_{1,1} = \frac{2}{3}, \theta_{2,1} = -\frac{1}{2}, \theta_{2,2} = \frac{1}{2}, \theta_{c1} = -\frac{1}{4}, \theta_{c2} = \frac{3}{4}, \theta_{c3} = \frac{1}{2}$.
    \item RK3 (\#2) $\theta_{1,1} = \frac{2}{3}, \theta_{2,1} = \frac{1}{6}, \theta_{2,2} = \frac{1}{2}, \theta_{c1} = \frac{1}{4}, \theta_{c2} = \frac{1}{4}, \theta_{c3} = \frac{1}{2}$.
    \item RK3 (\#3) $\theta_{1,1} = \frac{1}{2}, \theta_{2,1} = -1, \theta_{2,2} = 2, \theta_{c1} = \frac{1}{6}, \theta_{c2} = \frac{2}{3}, \theta_{c3} = \frac{1}{6}$.
\end{itemize}
We define the error from RK3 as $E_{RK}$ and from RK-NN as $E_{NN}$, then we have relative error $\frac{E_{NN}}{E_{RK}}$. RK-NN performs better if relative error is smaller than 1. 
\cref{fig:diff_rk3} illustrates the relative errors by RK3 methods with distinct coefficients are almost the same and our RK-NN has a better performance in the training range of time step $h \in (0.01,0.1)$.

\begin{figure}[htb!]
\centering
	\begin{subfigure}{0.49\textwidth}
		\centering
		\includegraphics[width=1.0\linewidth]{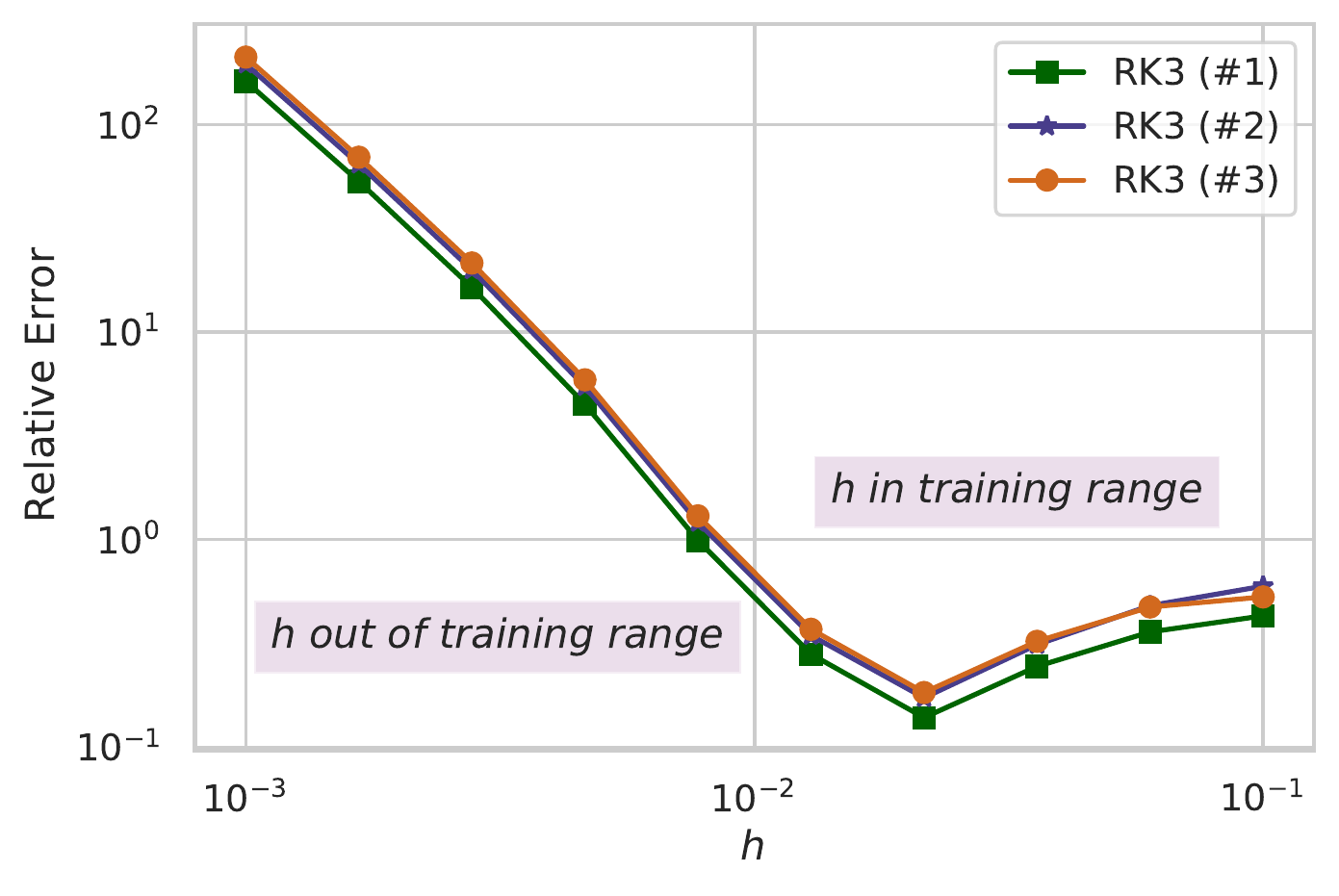}
		\caption{Van der Pol oscillator.}
		\label{fig:vdp_diff_rk3}
	\end{subfigure}
	\begin{subfigure}{0.49\textwidth}
		\centering
		\includegraphics[width=1.0\linewidth]{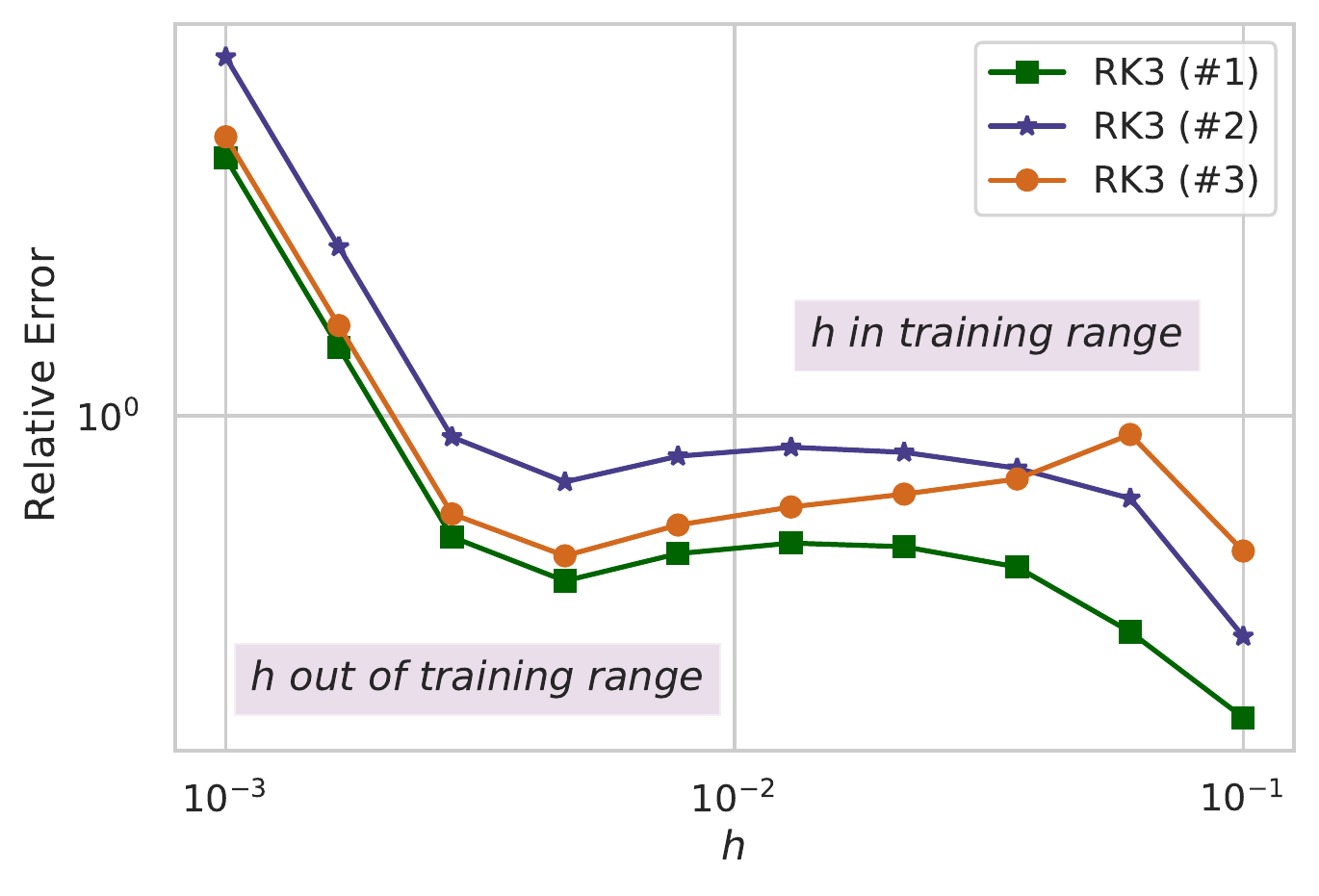}
		\caption{Brusselator.}
		\label{fig:bru_diff_rk3}
	\end{subfigure}
	\caption{Comparison of RK-NN results with different parameterized RK3 methods.}
	\label{fig:diff_rk3}
\end{figure}

\section*{SM5. Training Complexity for ODE systems of different dimension}
\label{training_complexity}

In \cref{fig:Training_Complexity}, we investigate empirically the scalability of our method during training.
We plot the wall-clock time required for one epoch of training on
an example linear family and a nonlinear family as the dimension $d$
of the ODE increases.
We observe that the training cost per epoch has a scaling between $O(d)$ and $O(d^{2})$, meaning that our method
can be effectively applied to moderately high dimensional systems.
In the problems we have tested, the number of epochs required to reach a specified testing accuracy
does not increase significantly with ODE dimension,
and is in general problem dependent
(See \cref{table:no_of_epochs}).

\begin{figure}[htb!]
	\centering	\includegraphics[width=0.6\textwidth]{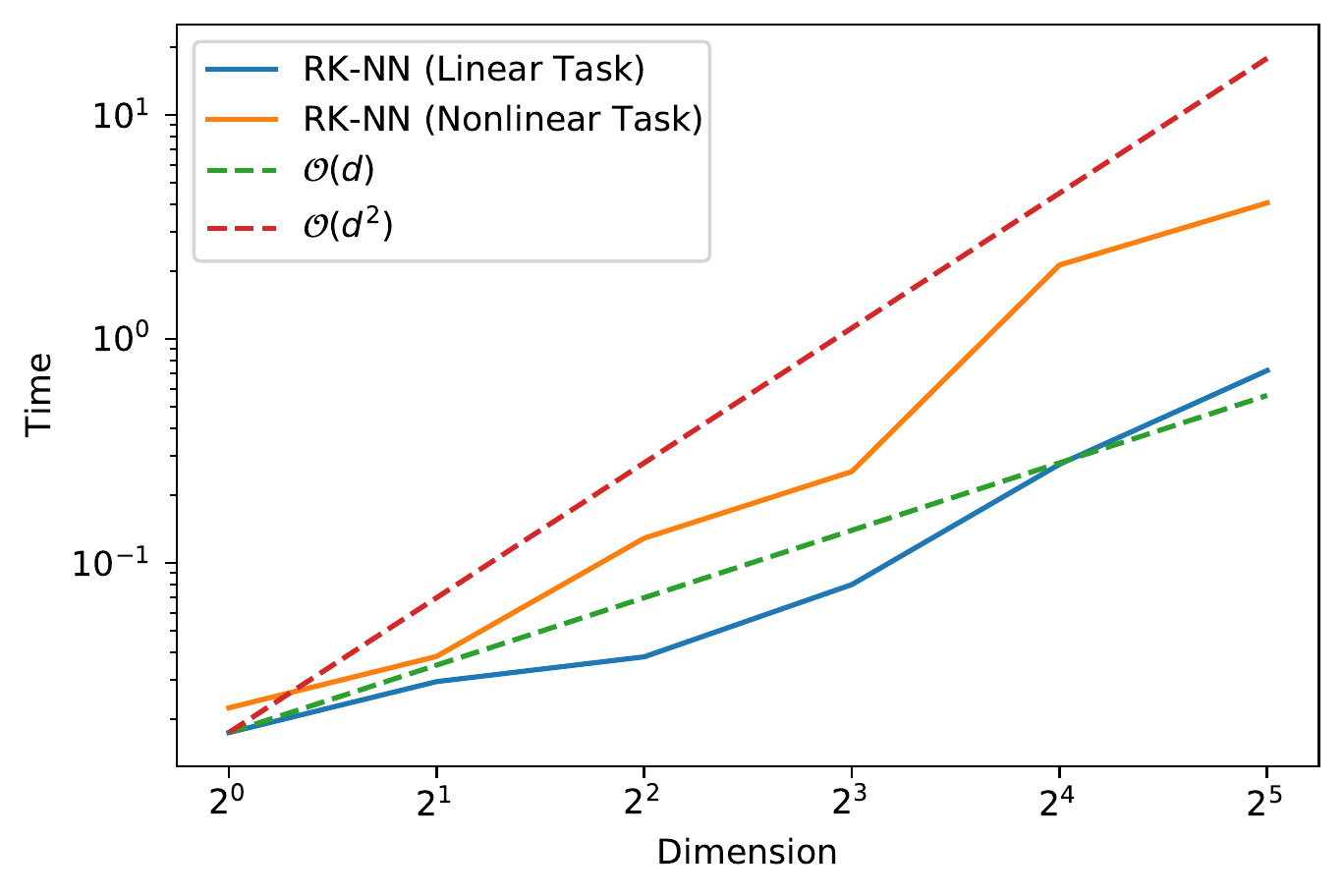}
	\caption{Training time per epoch for different dimensions. The equation of linear tasks is $\frac{\dd \y}{\dd t} = A \y, \y(0) = \y_{0}$, where $A \in \R^{d \times d}$ is a square matrix. The elements of $A$ are independently and identically distributed with
	$A_{ij} \sim  U(-\frac{1}{\sqrt{d}},-\frac{2}{\sqrt{d}})$ and $\y_{0} \sim U(-3,3)$.
	The equation of nonlinear tasks is $\frac{\dd \y}{\dd t} = B \y^{2}, \y(0) = \y_{0}$. $B \in \R^{d \times d}$ is a diagonal matrix whose diagonal elements are randomly generated from different distributions.
	The $i$-th element $B_{ii} \sim U(-2+0.05(i-1), -2+0.05(i+1))$ and $\y_{0} \sim U(1,3)$. In this figure, the x-axis is the dimension $d$ and the y-axis is the wall-clock time required for each training epoch.}
	\label{fig:Training_Complexity}
\end{figure}

\begin{table}[htb!]
    \caption{Comparison of the number of training epochs for ODE systems of different dimensions shown in \cref{fig:Training_Complexity}}
    \label{table:no_of_epochs}
    \vspace{20pt}
    \centering
    \begin{tabular}{p{3.5cm}p{0.8cm}p{0.8cm}p{0.8cm}p{0.8cm}p{0.8cm}p{0.8cm}p{0.8cm}}
        \hline
         Dimension & 1 & 2 & 4 & 8 &  16  & 32 \\
        \hline
        Linear Task & 201 & 201 & 196 & 191 & 187 & 200 \\
        \hline    
        Nonlinear Task & 2336 & 1831 & 1506 & 1511 & 1577 &  1527\\
        \hline
    \end{tabular}
\end{table}

\end{document}